\documentclass[10pt]{amsart}
\usepackage{amsmath,amscd}
\usepackage{amsbsy}
\usepackage{amssymb}
\usepackage{amscd,amsthm}
\usepackage[all,cmtip]{xy}

\newtheorem{thm}{Theorem}
\numberwithin{thm}{section}

\newtheorem{prop}[thm]{Proposition}
\newtheorem{cor}[thm]{Corollary}
\newtheorem{exam}[thm]{Example}
\newtheorem{rema}[thm]{Remark}

\newtheorem{defi}[thm]{Definition}

\newtheorem*{thm2}{Theorem}
\newtheorem*{cor2}{Corollary}
\newtheorem*{prop2}{Proposition}

\newtheorem*{que2}{Question}

\begin{document}
\begin{center}
\huge{Non-existence of full exceptional collections on twisted flags and categorical representability via noncommutative motives}\\[1cm]
\end{center}

\begin{center}

\large{Sa$\mathrm{\check{s}}$a Novakovi$\mathrm{\acute{c}}$}\\[0,4cm]
%\small{\emph{To B. with love}}\\[0,6cm]
{\small December 2019}\\[0,3cm]
\end{center}

\noindent{\small \textbf{Abstract}. 
In this paper we prove that a finite product of Brauer--Severi varieties is categorical representable in dimension zero if and only if it admits a $k$-rational point if and only if it is rational over $k$. The same is true for certain isotropic involution varieties over a field $k$ of characteristic different from two. 
For finite products of generalized Brauer--Severi varieties, categorical representability in dimension zero is equivalent to the existence of a full exceptional collection. In this case however categorical representability in dimension zero is not equivalent to the existence of a rational point. We also show that non-trivial twisted flags of classical type $A_n$ and $C_n$ cannot have full exceptional collections, enlarging in this way the set of previous known examples. Finally, we determine the categorical representability dimension $\mathrm{rdim}(X)$ for generalized Brauer--Severi varieties of index $\leq 3$ and for certain twisted forms of smooth quadrics (involution varieties).\\

%\begin{center}
%\tableofcontents
%\end{center}
\section{Introduction}
Since the early works by Beilinson, Bondal and Kapranov on exceptional collections in geometric categories it has been conjectured that projective homogeneous spaces over algebraically closed fields of characteristic zero should have full exceptional collections. The best approximation at the moment is the paper of Kuznetsov and Polishchuk \cite{KU1S} in which a uniform construction of exceptional collections has been proposed. However, it still remains to show that these collections are full. Quite recently, it has been proved that some of the exceptionel collection from \cite{KU1S} are indeed full \cite{ANS}. But there is an interesting problem appearing if the base field of the homogeneous spaces is not algebraically closed. Homogeneous spaces can be defined over $\mathbb{Z}$ and it is therefore natural to ask whether there exist full exceptional collections in this general case. It is also natural to consider twisted forms of homogeneous spaces and to study whether or not they admit full exceptional collections. And indeed, Raedschelders \cite{RAES} proved recently that non-split Brauer--Severi varieties do not have full \'etale exceptional collections and provided in this way a positive answer to the conjecture that non-split Brauer--Severi varieties do not admit full exceptional collections (see \cite{NO2S}).

The aim of the present paper is to discuss the existence of (full) exceptional collections on certain twisted flags ${_\gamma}X$ and its relation to the existence of a rational point, respectively the rationality of the given variety ${_\gamma}X$. All the proofs of the results concerning the (non)-existence of exceptional collections make use of noncommutative motives and are consequences of the main results in \cite{TAS} and the idea from the note \cite{RAES}. Therefore, the results for Brauer--Severi varieties are somehow known as they naturally extend \cite{RAES}, but, to our best knowledge, stated nowhere. Nonetheless, we want to give the proofs, adding to the literature. The results for twisted quadrics however are new (see Theorems 5.6 and 5.7 and Corollary 5.5). The interesting observation of the present paper however is, that the (non)-existence of full exceptional collections for the varieties under consideration is related to the existence of $k$-rational points and to the rationality of the given variety. So the main results in this context will be Theorem 6.3, Corollaries 6.9 and 6.11 and Propositions 6.12 and 6.13 below.   
\begin{thm2}[Proposition 5.1, Proposition 5.3]
	Let $X=X_1\times\cdots\times X_n$ be a finite product, where all the factors are either Brauer--Severi varieties over an arbitrary field or generalized Brauer--Severi varieties over a field of characteristic zero. Then $X$ admits a full exceptional collection if and only if all the $X_i$ are split, i.e, if all $X_i$ are either projective spaces or Grassmannians.
\end{thm2} 
\begin{thm2}[Corollary 5.5, Theorem 5.7]
	Let $X=X_1\times\cdots\times X_n$ be a finite product of twisted forms of smooth quadrics over $k$ ($\mathrm{char}(k)\neq 2$) where all factors $X_i$ are either associated with involution algebras $(A_i,\sigma_i)$ of orthogonal type having trivial discriminants $\delta(A_i,\sigma_i)$ or with involution algebras $(A_i,\sigma_i)$ of orthogonal type over $\mathbb{R}$. Then $X$ admits a full exceptional collection if and only if all the $X_i$ are split, i.e, if all $X_i$ are smooth projective quadrics.
\end{thm2}
Using noncommutative motives again and exploiting general facts from \cite{PAS} and \cite{MPWS} we show that non-trivial twisted forms of homogeneous spaces of type $A_n$ and $C_n$ can not have full exceptional collections.
\begin{cor2}[Corollary 5.10]
	Let $G_i$ be a split simply connected simple algebraic group of classical type $A_n$ or $C_n$ ($n\neq 4$) over $k$ and $\gamma_i\colon \mathrm{Gal}(k^s|k)\rightarrow G_i(k^s)$, $1\leq i\leq n$, 1-cocycles. Given parabolic subgroups $P_i\subset G_i$, let ${_{\gamma_i}}(G_i/P_i)$ be the twisted forms of $G_i/P_i$. If $X={_{\gamma_1}}(G_1/P_1)\times...\times {_{\gamma_n}}(G_n/P_n)$ admits a full exceptional collection, then all 1-cocycles $\gamma_i$ must be trivial.  In other words, non-trivial twisted flags of the considered type do not have full exceptional collections.
\end{cor2}

Closely related to the problem of the existence of full exceptional collections is the question whether the considered variety admits a $k$-rational point. A potential measure for rationality was introduced by Bernardara and Bolognesi \cite{BBS} with the notion of categorical representability. We use the definition given in \cite{AB1S}. A $k$-linear triangulated category $\mathcal{T}$ is said to be \emph{representable in dimension $m$} if there is a semiorthogonal decomposition (see Section 3 for the definition) $\mathcal{T}=\langle \mathcal{A}_1,...,\mathcal{A}_n\rangle$ and for each $i=1,...,n$ there exists a smooth projective connected variety $Y_i$ with $\mathrm{dim}(Y_i)\leq m$, such that $\mathcal{A}_i$ is equivalent to an admissible subcategory of $D^b(Y_i)$. We use the following notations
\begin{eqnarray*}
	\mathrm{rdim}(\mathcal{T}):=\mathrm{min}\{m\mid \mathcal{T}\  \textnormal{is representable in dimension m}\},
\end{eqnarray*}
whenever such a finite $m$ exists. Let $X$ be a smooth projective $k$-variety. One says $X$ is \emph{representable in dimension} $m$ if $D^b(X)$ is representable in dimension $m$. We will use the following notation:
\begin{eqnarray*}
	\mathrm{rdim}(X):=\mathrm{rdim}(D^b(X)),\thickspace \mathrm{rcodim}(X):= \mathrm{dim}(X)-\mathrm{rdim}(X).
\end{eqnarray*}
Note that when the base field $k$ of a variety $X$ is not algebraically closed, the existence of $k$-rational points on $X$ is a major open question in arithmetic geometry. We recall the following question that was asked by H. Esnault and stated in \cite{ABS}.
\begin{que2}[H. Esnault]
	Let $X$ be a smooth projective variety over $k$. Can the bounded derived category $D^b(X)$ detect the existence of a $k$-rational point?
\end{que2}

\begin{thm2}[Theorem 6.3]
	If $X$ is a finite product of Brauer--Severi varieties, then the following are equivalent:
	\begin{itemize}
		\item[(\textbf{i})] $\mathrm{rdim}(X)=0$.
		\item[(\textbf{ii})] $X$ admits a full exceptional collection.
		\item[(\textbf{iii})] $X$ is rational over $k$.
		\item[(\textbf{iv})] $X$ admits a $k$-rational point.
	\end{itemize} 
\end{thm2}

Note that this theorem is a generalization of Proposition 6.1 in \cite{ABS} and that it is proved in a completely different way than the result in \cite{ABS}. For the product of generalized Brauer--Severi varieties we will show:
\begin{thm2}[Theorem 6.5, Corollary 6.6]
	Let $X$ be a finite product of generalized Brauer--Severi varieties over a field $k$ of characteristic zero. Then $\mathrm{rdim}(X)=0$ if and only if $X$ splits as the finite product of Grassmannians over $k$.
\end{thm2}
We want to point out that Theorem 6.5 in not equivalent to the statement that $X$ admits a $k$-rational point (see Remark 6.7 for an explanation).
Moreover, for the finite product of certain twisted forms of smooth quadrics we show the following:
\begin{thm2}[Corollary 5.5, Theorem 5.7]
	Let $X=X_1\times\cdots\times X_n$ be a finite product of twisted forms of smooth quadrics over $k$ where all factors $X_i$ are either associated with involution algebras $(A_i,\sigma_i)$ of orthogonal type over a field $k$ ($\mathrm{char}(k)\neq 2$) with trivial discriminants $\delta(A_i,\sigma_i)$ or with involution algebras $(A_i,\sigma_i)$ of orthogonal type over $\mathbb{R}$. Then $\mathrm{rdim}(X)=0$ if and only if $X$ splits as the product of smooth quadrics, i.e. if all $(A_i,\sigma_i)$ split. In particular, $X$ admits a full exceptional collection if and only if $X$ splits as the product of smooth quadrics
\end{thm2}
If all the factors in the product of twisted quadrics are isotropic, we find:
\begin{thm2}[Corollary 6.9, Corollary 6.11]
	Let $X=X_1\times\cdots\times X_n$ be a finite product of twisted forms of smooth quadrics over $k$ where all factors $X_i$ are either associated with isotropic involution algebras $(A_i,\sigma_i)$ of orthogonal type over a field $k$ ($\mathrm{char}(k)\neq 2$) with trivial discriminants $\delta(A_i,\sigma_i)$ or with isotropic involution algebras $(A_i,\sigma_i)$ of orthogonal type over $\mathbb{R}$. Then the following are equivalent:
	\begin{itemize}
		\item[(\textbf{i})]	$\mathrm{rdim}(X)=0$.
		\item[(\textbf{ii})] $X$ admits a full exceptional collection.
		\item[(\textbf{iii})] $X$ is rational over $k$.
		\item[(\textbf{iv})] $X$ admits a $k$-rational point.
	\end{itemize}
\end{thm2}
The equivalence of (ii) and (iii) (resp. (iv)) however, is not true in general. Indeed there exists a smooth anisotropic quadric without rational point that admits a full exceptional collection (see \cite{BTS}, proof of Proposition 1.7). Note that a smooth anisotropic quadric is rational if and only if it admits a rational point (see \cite{COS}, Theorem 1.1).

We want to stress that the existence of rational points seems, in general, not to be related to categorical representability in dimension zero. For instance, an elliptic curve over a number field is categorical representable in dimension one (see \cite{AB1S}) although it has rational points. Indeed, the rationality of a given variety $X$ seems to be related to categorical representability in codimension 2. We recall the following question which was formulated in \cite{MBTS}:
\begin{que2}
	Let $X$ be a smooth projective variety over $k$ of dimension at least 2. Suppose $X$ is $k$-rational. Do we have $\mathrm{rcodim}(X)\geq 2$ ?
\end{que2} 
There are several results suggesting that this question has a positive answer, see \cite{MBTS}, \cite{AB1S} and references therein. In this context, Theorem 6.3 and Corollaries 6.9 and 6.11 from above have the following consequence:
\begin{prop2}[Proposition 6.12]
	Let $X=X_1\times\cdots \times X_n$ be a finite product over $k$ ($\mathrm{char}(k)\neq 2$) of dimension at least two. Assume $X_i$ is either a Brauer--Severi variety or a twisted form of a quadric associated with an isotropic involution algebra of orthogonal type with trivial discriminant. If $X$ is $k$-rational, then $\mathrm{rcodim}(X)\geq 2$.
\end{prop2}
\begin{prop2}[Proposition 6.13]
	Let $X=X_1\times\cdots \times X_n$ be a finite product over $\mathbb{R}$ of dimension at least two. Assume $X_i$ is either a Brauer--Severi variety or a twisted form of a quadric associated with an isotropic involution algebra of orthogonal type. If $X$ is $k$-rational, then $\mathrm{rcodim}(X)\geq 2$.
\end{prop2}
In the special case where $X$ is a Brauer--Severi variety or a twisted form of a smooth quadric associated to an isotropic central simple algebra with involution of orthogonal type, $k$-rationality is equivalent to the existence of a $k$-rational point. So it is clear that Theorem 6.3 and Corollaries 6.9 and 6.11 reflect a very special behavior for the varieties under examination. 

The number $\mathrm{rdim}$ is closely related to the \emph{motivic categorical dimension} which is introduced in \cite{MBS} for arbitrary fields. In $\mathrm{char}(k)=0$, motivic categorical dimension of a smooth projective variety $X$ can be defined as the smallest $d$ such that $\mu([X])$ lies in $PT_d(k)$. Here $PT(k)$ denotes the Grothendieck ring of dg categories (see \cite{BLLS}), $PT_d(k)\subset PT(k)$ the additive subgroup generated by the smallest saturated monoid containing classes of pretriangulated dg categories of categorical dimension at least $d$ and $\mu\colon K_0(Var(k))\rightarrow PT(k)$ the motivic measure defined by Bondal, Larsen and Lunts in \cite{BLLS} (in $\mathrm{char}(k)=0$). If one denotes the motivic categorical dimension of $D^b(X)$ by $\mathrm{mcd}(X)$, one has $\mathrm{mcd}(X)\leq \mathrm{rdim}(X)$ (see \cite{MBS}). And a natural problem is then to calculate the numbers $\mathrm{mcd}(X)$  and $\mathrm{rdim}(X)$ and find conditions for which $X$ equality $\mathrm{mcd}(X)=\mathrm{rdim}(X)$ holds. And indeed, results from \cite{MBS} suggest that motivic categorical dimension can be used to define a birational invariant. It turns out that determining $\mathrm{rdim}(X)$ for $X$ a non-split (generalized) Brauer--Severi or an involution variety is indeed a challenging problem. We determine $\mathrm{rdim}(X)$ in the case of generalized Brauer--Severi varieties of index $\leq 3$ and generalize in this way an earlier result \cite{NO3S}, Theorem 1.4.
\begin{thm2}[Theorem 6.15]
	Let $X=\mathrm{BS}(r,A)$ be a generalized Brauer--Severi variety over a field $k$ of characteristic zero with $\mathrm{ind}(A)\leq 3$. Then $\mathrm{rdim}(X)=\mathrm{ind}(A)-1$. In particular $\mathrm{mcd}(X)\leq \mathrm{ind}(A)-1$. 	
\end{thm2}
\noindent
For involution varieties we obtain:
\begin{thm2}[Theorem 6.16]
	Let $X$ be one of the following varieties:
	\begin{itemize}
		\item[(\textbf{i})] a twisted quadric associated to a central simple algebra $(A,\sigma)$ with involution of orthogonal type having trivial discriminant $\delta(A,\sigma)$. 
		\item[(\textbf{ii})] a twisted quadric associated to a central simple $\mathbb{R}$-algebra $(A,\sigma)$ with involution of orthogonal type.
	\end{itemize} 
	Then $\mathrm{rdim}(X)\leq \mathrm{ind}(A)-1$. Moreover, if $\mathrm{ind}(A)\leq 3$, then $\mathrm{rdim}(X)= \mathrm{ind}(A)-1$. In particular, $\mathrm{mcd}(X)\leq \mathrm{ind}(A)-1$.
\end{thm2}

{\small \textbf{Acknowledgement}. Nearly all the proofs make use of noncommutative motives and I thank Theo Raedschelders for his explanations in this respect and notes on literature. I also like to thank Marcello Bernardara for useful comments. %Then I am grateful to the referee for careful reading and his helpful comments which helped to improve the paper. 
I would like to thank the Heinrich Heine University for financial support via the SFF-grant.}\\

{\small \textbf{Conventions}. Throughout this work $k$ denotes an arbitrary ground field and $k^s$ and $\bar{k}$ a separable respectively algebraic closure.} %Furthermore, any locally free sheaf is assumed to be of finite rank and will be called vector bundle.

\section{Examples of twisted forms of homogeneous varieties}

As references we use \cite{MPWS}, \S 1 and \cite{KNUS}, Chapter VI. By an algebraic group over the field $k$ we will always mean an \emph{affine} algebraic group. Let $G$ be an algebraic group over $k$ and $X$ an algebraic variety such that $G$ acts on $X$ over $k$. The variety $X$ is then called \emph{$G$-variety}. For a closed (and reduced) subgroup $H$ of $G$ one has the associated \emph{homogeneous $G$-variety} $Y=G/H$, which is also called a \emph{homogeneous space}. Now let $G$ be a semi-simple (so connected) algebraic group. A projective $G$-variety $X$ is called \emph{twisted flag} if $X\otimes_k k^s\simeq (G\otimes_k k^s)/P$ where $P$ is a closed subgroup of $G\otimes_k k^s$. As $(G\otimes_k k^s)/P$ is projective, $P$ must be a parabolic subgroup. Any twisted flag is a smooth, absolutely irreducible and reduced variety. 

It is well-known that the twisted forms of a homogeneous space $G/P$ are in one-to-one correspondence with elements in $H^1(k,\mathrm{Aut}(G\otimes_k k^s)/P))$. Let $\bar{G}$ denote the adjoint group $G/Z(G)$. The $\mathrm{Gal}(k^s|k)$-group homomorphism $\bar{G}(k^s)\rightarrow \mathrm{Aut}((G\otimes_k k^s)/P)$ induces a map of pointed sets
\begin{eqnarray*}
	\alpha\colon H^1(k,\bar{G}(k^s))\longrightarrow H^1(k,\mathrm{Aut}((G\otimes_k k^s)/P)).
\end{eqnarray*}
For a twisted form we write ${_\gamma}X$ for ${_{\alpha(\gamma)}}X$ where $X=G/P$.
%Let $G$ be a semi-simple (linear) algebraic group over a field $k$ and $P\subset G$ a parabolic subgroup. It is well known that out of this data one can construct a projective homogeneous space $G/P$.  An algebraic group $G'$ is called \emph{twisted form} of $G$ if $G'\otimes_k k^s\simeq G\otimes_k k^s$. The set of isomorphism classes of twisted forms of $G$ is in one-to-one correspondence with $H^1(k,\mathrm{Aut}(G\otimes_k k^s))$. Denote by $\bar{G}$ the adjoint group, then the natural homomorphism $\bar{G}(k^s)\rightarrow \mathrm{Aut}(G\otimes_k k^s)$ induces a map
%\begin{eqnarray*}
%\alpha\colon H^1(k,\bar{G}(k^s))\longrightarrow H^1(k,\mathrm{Aut}(G\otimes_k k^s)).
%\end{eqnarray*}
%A twisted form $G'$ of $G$ is called an \emph{inner form} if the cocycle corresponding to $G'$ belongs to the image of $\alpha$.
So in the present work we will always take some 1-cocycle $\gamma\colon \mathrm{Gal}(k^s|k)\rightarrow G(k^s)$, the projective homogeneous space $X=G/P$ and its twisted form ${_\gamma}X$.  Note that ${_\gamma}X\otimes_k k^s\simeq (G/P)\otimes_k k^s$. Furthermore, let $\widetilde{G}$ and $\widetilde{P}$ be the universal covers of $G$ and $P$ respectively. Denote by $R(\widetilde{G})$ and $R(\widetilde{P})$ the associated representation rings and by $\widetilde{Z}\subset \widetilde{G}$ the center of $\widetilde{G}$. Finally, let $\mathrm{Ch}:=\mathrm{Hom}(\widetilde{Z},\mathbb{G}_m)$ be the character group. Under these notations we can give some examples of twisted forms of homogeneous spaces which will occur quite frequently throughout this work.

\begin{exam}
	\textnormal{Let $G=\mathrm{PGL}_n$. In this case we have $\widetilde{G}=\mathrm{SL}_n$ and $\widetilde{Z}\simeq \mu_n$. Then $\mathrm{Ch}\simeq \mathbb{Z}/n\mathbb{Z}$. Let $a\in k^{\times}$ and $c\in \mathrm{GL}_{n-1}$ and consider the following parabolic subgroup
		\begin{eqnarray*}
			\widetilde{P}=\{ \begin{pmatrix} 
				a&b\\
				0&c
			\end{pmatrix} \big\vert a\cdot\mathrm{det}(c)=1\}\subset \mathrm{SL}_n.
		\end{eqnarray*} The associated projective homogeneous variety is $\widetilde{G}/\widetilde{P}\simeq G/P\simeq \mathbb{P}^{n-1}_k$. Let $\gamma\colon \mathrm{Gal}(k^s|k)\rightarrow \mathrm{PGL}_n(k^s)$ be a 1-cocycle, then the twisted form ${_\gamma}\mathbb{P}^{n-1}$ is called \emph{Brauer--Severi variety}.}
\end{exam}
\begin{exam}
	\textnormal{Let $G=\mathrm{PGL}_n$ as in Example 2.1. We fix a number $1\leq d\leq n-1$ and let $a\in \mathrm{GL}_d$ and $c\in \mathrm{GL}_{n-d}$. Consider the parabolic subgroup 
		\begin{eqnarray*}
			\widetilde{P}=\{ \begin{pmatrix} 
				a&b\\
				0&c
			\end{pmatrix} \big\vert a\cdot\mathrm{det}(c)=1\}\subset \mathrm{SL}_n.
		\end{eqnarray*} The associated projective homogeneous variety is $\widetilde{G}/\widetilde{P}\simeq G/P\simeq \mathrm{Grass}_k(d,n)$. Given a 1-cocycle $\gamma\colon\mathrm{Gal}(k^s|k)\rightarrow \mathrm{PGL}_n(k^s)$, the twisted form ${_\gamma}\mathrm{Grass}_k(d,n)$ is called \emph{generalized Brauer--Severi variety}.}
\end{exam}
\begin{exam}
	\textnormal{Let $G=\mathrm{PSO}_n$ with $n$ even. In this case $\widetilde{G}= \mathrm{Spin}_n$. Consider the action of $G$ on $\mathbb{P}^{n-1}$ given by projective linear transformations. We write $P\subset G$ for the stabilizer of the point $[1:0:\cdots :0]$. The projective homogeneous variety $G/P$ is a smooth quadric hypersurface $Q\subset \mathbb{P}^{n-1}$. Given a 1-cocycle $\gamma\colon \mathrm{Gal}(k^s|k)\rightarrow \mathrm{PSO}_n(k^s)$, it is well known that $\gamma$ determines a central simple $k$-algebra $A$ with an involution $\sigma$ of orthogonal type. The associated twisted homogeneous space ${_\gamma}(G/P)$ is a twisted form of the quadric $G/P$.}
\end{exam}
In all examples from above the twisted forms can be described in terms of central simple $k$-algebras.
Recall that a finite-dimensional $k$-algebra $A$ is called \emph{central simple} if it is an associative $k$-algebra that has no two-sided ideals other than $0$ and $A$ and if its center equals $k$. If the algebra $A$ is a division algebra it is called \emph{central division algebra}. Note that $A$ is a central simple $k$-algebra if and only if there is a finite field extension $k\subset L$, such that $A\otimes_k L \simeq M_n(L)$. This is also equivalent to $A\otimes_k \bar{k}\simeq M_n(\bar{k})$. An extension $k\subset L$ such that $A\otimes_k L\simeq M_n(L)$ is called splitting field for $A$. 

The \emph{degree} of a central simple algebra $A$ is defined to be $\mathrm{deg}(A):=\sqrt{\mathrm{dim}_k A}$. According to the \emph{Wedderburn Theorem}, for any central simple $k$-algebra $A$ there is an unique integer $n>0$ and a division $k$-algebra $D$ such that $A\simeq M_n(D)$. The division algebra $D$ is also central and unique up to isomorphism. The degree of the unique central division algebra $D$ is called the \emph{index} of $A$ and is denoted by $\mathrm{ind}(A)$. Two central simple algebras $A$ and $B$ are said to be \emph{Brauer-equivalent} if there are positive integers $r,s$ such that $M_r(A)\simeq M_s(B)$. 

Note that a Brauer--Severi variety of dimension $n$ can also be defined as a scheme $X$ of finite type over $k$ such that $X\otimes_k L\simeq \mathbb{P}^n$ for a finite field extension $k\subset L$. A field extension $k\subset L$ for which $X\otimes_k L\simeq \mathbb{P}^n$ is called \emph{splitting field} of $X$. Clearly, $k^s$ and $\bar{k}$ are splitting fields for any Brauer--Severi variety. In fact, every Brauer--Severi variety always splits over a finite Galois extension. It follows from descent theory that $X$ is projective, integral and smooth over $k$. Via non-commutative Galois cohomology, Brauer--Severi varieties of dimension $n$ are in one-to-one correspondence with central simple algebras $A$ of degree $n+1$. For details and proofs on all mentioned facts we refer to \cite{ARS} and \cite{GSS}.

To a central simple $k$-algebra $A$ one can also associate twisted forms of Grassmannians. Let $A$ be of degree $n$ and $1\leq d\leq n$. Consider the subset of $\mathrm{Grass}_k(d\cdot n, A)$ consisting of those subspaces of $A$ that are left ideals $I$ of dimension $d\cdot n$. This subset can be given the structure of a projective variety which turns out to be a generalized Brauer--Severi variety. It is denoted by $\mathrm{BS}(d,A)$. After base change to some splitting field $L$ of $A$ the variety $\mathrm{BS}(d,A)$ becomes isomorphic to $\mathrm{Grass}_L(d,n)$. If $d=1$ the generalized Brauer--Severi variety is the Brauer--Severi variety associated to $A$. Note that $\mathrm{BS}(d,A)$ is a Fano variety. For details see \cite{BLS}. 

To a central simple algebra $A$ of degree $n$ with involution $\sigma$ of the first kind over a field $k$ of $\mathrm{char}(k)\neq 2$ one can associate the \emph{involution variety} $\mathrm{IV}(A,\sigma)$. This variety can be described as the variety of $n$-dimensional right ideals $I$ of $A$ such that $\sigma(I)\cdot I=0$. If $A$ is split so $(A,\sigma)\simeq (M_n(k), q^*)$, where $q^*$ is the adjoint involution defined by a quadratic form $q$ one has $\mathrm{IV}(A,\sigma)\simeq V(q)\subset \mathbb{P}^{n-1}_k$. Here $V(q)$ is the quadric determined by $q$. By construction such an involution variety  $\mathrm{IV}(A,\sigma)$ becomes a quadric in $\mathbb{P}^{n-1}_L$ after base change to some splitting field $L$ of $A$. In this way the involution variety is a twisted form of a smooth quadric in the sense of Example 2.3. Recall from \cite{DTS} that a splitting field $L$ splits $A$ \emph{isotropically} if $(A,\sigma)\otimes_k L\simeq (M_n(L), q^*)$ with $q$ an isotropic quadratic form over $L$. Although the degree of $A$ is arbitrary, (when $\mathrm{char}(k)\neq 2$), the case where degree of $A$ is odd does not give anything new, since central simple algebras of odd degree with involution of the first kind are split (see \cite{KNUS}, Corollary 2.8). For details on the construction and further properties on involution varieties and the corresponding algebras we refer to \cite{DTS}. 

In fact, all the twisted flags from Examples 2.1, 2.2 and 2.3 appear as Brauer--Severi varieties, generalized Brauer--Severi varieties or as twisted smooth quadrics and are associated to some central simple algebra $A$ in the sense described above.

\section{Exceptional collections and semiorthogonal decompositions} 
Let $\mathcal{D}$ be a triangulated category and $\mathcal{C}$ a triangulated subcategory. The subcategory $\mathcal{C}$ is called \emph{thick} if it is closed under isomorphisms and direct summands. For a subset $A$ of objects of $\mathcal{D}$ we denote by $\langle A\rangle$ the smallest full thick subcategory of $\mathcal{D}$ containing the elements of $A$. 
%Furthermore, we define $A^{\perp}$ to be the subcategory of $\mathcal{D}$ consisting of all objects $M$ such that $\mathrm{Hom}_{\mathcal{D}}(E[i],M)=0$ for all $i\in \mathbb{Z}$ and all elements $E$ of $A$. We say that $A$ \emph{generates} $\mathcal{D}$ if $A^{\perp}=0$. Now assume $\mathcal{D}$ admits arbitrary direct sums. An object $B$ is called \emph{compact} if $\mathrm{Hom}_{\mathcal{D}}(B,-)$ commutes with direct sums. Denoting by $\mathcal{D}^c$ the subcategory of compact objects we say that $\mathcal{D}$ is \emph{compactly generated} if the objects of $\mathcal{D}^c$ generate $\mathcal{D}$. One has the following important theorem (see \cite{BV}, Theorem 2.1.2).
%\begin{thm}
%Let $\mathcal{D}$ be a compactly generated triangulated category. Then a set of objects $A\subset \mathcal{D}^c$ generates $\mathcal{D}$ if and only if $\langle A\rangle=\mathcal{D}^c$.  
%\end{thm}
For a smooth projective variety $X$ over $k$, we denote by $D^b(X)$ the bounded derived category of coherent sheaves on $X$. Moreover, if $B$ is an associated $k$-algebra, we write $D^b(B)$ for the bounded derived category of finitely generated left $B$-modules.

%$D(\mathrm{Qcoh}(X))$ the derived category of quasicoherent sheaves on $X$. The bounded derived category of coherent sheaves is denoted by $D^b(X)$. %Note that $D(\mathrm{Qcoh}(X))$ is compactly generated with compact objects being all of $D^b(X)$. For details on generating see \cite{BV}.
\begin{defi}
	\textnormal{Let $A$ be a division algebra over $k$, not necessarily central. An object $\mathcal{E}^{\bullet}\in D^b(X)$ is called \emph{$A$-exceptional} if $\mathrm{End}(\mathcal{E}^{\bullet})=A$ and $\mathrm{Hom}(\mathcal{E}^{\bullet},\mathcal{E}^{\bullet}[r])=0$ for $r\neq 0$. By \emph{generalized exceptional object}, we mean $A$-exceptional for some division algebra $A$ over $k$. %The $w$ in w-exceptional stands for weak (see \cite{OR}). 
		If $A=k$, the object $\mathcal{E}^{\bullet}$ is called \emph{exceptional}.} 
\end{defi}
\begin{defi}
	\textnormal{A totally ordered set $\{\mathcal{E}^{\bullet}_1,...,\mathcal{E}^{\bullet}_n\}$ of generalized exceptional objects on $X$ is called an \emph{generalized exceptional collection} if $\mathrm{Hom}(\mathcal{E}^{\bullet}_i,\mathcal{E}^{\bullet}_j[r])=0$ for all integers $r$ whenever $i>j$. A generalized exceptional collection is \emph{full} if $\langle\{\mathcal{E}^{\bullet}_1,...,\mathcal{E}^{\bullet}_n\}\rangle=D^b(X)$ and \emph{strong} if $\mathrm{Hom}(\mathcal{E}^{\bullet}_i,\mathcal{E}^{\bullet}_j[r])=0$ whenever $r\neq 0$. If the set $\{\mathcal{E}^{\bullet}_1,...,\mathcal{E}^{\bullet}_n\}$ consists of exceptional objects it is called \emph{exceptional collection}.}
\end{defi}
%Notice that the direct sum of objects forming a full strong w-exceptional (resp. separable-exceptional) collection is a tilting object in the sense of Definition 3.2. 
\begin{rema}
	\textnormal{If the ring $A$ in Definition 3.1 is required to be a semisimple algebra, the object is also called \emph{semi-exceptional object} in the literature (see \cite{ORS}). Consequently, one can also define (full) semi-exceptional collections. We also want to mention that a general exceptional object is also called weak exceptional object in the literature (see \cite{ORS}, Definition 1.16)}
\end{rema}
\begin{exam}
	\textnormal{Let $\mathbb{P}^n$ be the projective space and consider the ordered collection of invertible sheaves $\{\mathcal{O}_{\mathbb{P}^n}, \mathcal{O}_{\mathbb{P}^n}(1),...,\mathcal{O}_{\mathbb{P}^n}(n)\}$. In \cite{BES} Beilinson showed that this is a full strong exceptional collection.}
\end{exam}
\begin{exam}
	\textnormal{Let $X=\mathbb{P}^1\times\mathbb{P}^1$. Then $\{\mathcal{O}_X,\mathcal{O}_X(1,0), \mathcal{O}_X(0,1), \mathcal{O}_X(1,1)\}$ is a full strong exceptional collection on $X$. We write $\mathcal{O}_X(i,j)$ for $\mathcal{O}(i)\boxtimes\mathcal{O}(j)$.}
\end{exam}
The notion of a full exceptional collection is a special case of what is called a semiorthogonal decomposition of $D^b(X)$. Recall that a full thick triangulated subcategory $\mathcal{D}$ of $D^b(X)$ is called \emph{admissible} if the inclusion $\mathcal{D}\hookrightarrow D^b(X)$ has a left and right adjoint functor. 
\begin{defi}
	\textnormal{Let $X$ be a smooth projective variety over $k$. A sequence $\mathcal{D}_1,...,\mathcal{D}_n$ of full admissible triangulated subcategories of $D^b(X)$ is called \emph{semiorthogonal} if $\mathcal{D}_j\subset \mathcal{D}_i^{\perp}=\{\mathcal{F}^{\bullet}\in D^b(X)\mid \mathrm{Hom}(\mathcal{G}^{\bullet},\mathcal{F}^{\bullet})=0$, $\forall$ $ \mathcal{G}^{\bullet}\in\mathcal{D}_i\}$ for $i>j$. Such a sequence defines a \emph{semiorthogonal decomposition} of $D^b(X)$ if the smallest thick full subcategory containing all $\mathcal{D}_i$ equals $D^b(X)$.}
\end{defi}
\noindent
For a semiorthogonal decomposition we write $D^b(X)=\langle \mathcal{D}_1,...,\mathcal{D}_n\rangle$.
\begin{rema}
	\textnormal{Let $\mathcal{E}^{\bullet}_1,...,\mathcal{E}^{\bullet}_n$ be a full generalized exceptional collection on $X$. It is easy to verify that by setting $\mathcal{D}_i=\langle\mathcal{E}^{\bullet}_i\rangle$ one gets a semiorthogonal decomposition $D^b(X)=\langle \mathcal{D}_1,...,\mathcal{D}_n\rangle$.}
\end{rema}
\noindent
For a wonderful and comprehensive overview of the theory on semiorthogonal decompositions and its relevance in algebraic geometry we refer to \cite{KUS}.

\section{Recollections on noncommutative motives}
We refer to the book \cite{GTAS} (alternatively see \cite{TTS} and \cite{MTS} for a survey on noncommutative motives). Let $\mathcal{A}$ be a small dg category (for details see \cite{KELS}). The homotopy category $H^0(\mathcal{A})$ has the same objects as $\mathcal{A}$ and as morphisms $H^0(\mathrm{Hom}_{\mathcal{A}}(x,y))$. A source of examples is provided by schemes since the derived category of perfect complexes $\mathrm{perf}(X)$ of any quasi-projective scheme $X$ admits a canonical (unique) dg enhancement $\mathrm{perf}_{dg}(X)$ (see \cite{LOS}, Theorem 7.9). Denote by $\textbf{dgcat}$ the category of small dg categories. The \emph{opposite} dg category $\mathcal{A}^{op}$ has the same objects as $\mathcal{A}$ and $\mathrm{Hom}_{\mathcal{A}^{op}}(x,y):=\mathrm{Hom}_{\mathcal{A}}(y,x)$. A \emph{right $\mathcal{A}$-module} is a dg functor $\mathcal{A}^{op}\rightarrow C_{dg}(k)$ with values in the dg category $C_{dg}(k)$ of complexes of $k$-vector spaces. We write $C(\mathcal{A})$ for the category of right $\mathcal{A}$-modules. Recall form \cite{KELS} that the \emph{derived category} $D(\mathcal{A})$ of $\mathcal{A}$ is the localization of $C(\mathcal{A})$  with respect to quasi-isomorphisms. A dg functor $F\colon \mathcal{A}\rightarrow \mathcal{B}$ is called \emph{derived Morita equivalence} if the restriction of scalars functor $D(\mathcal{B})\rightarrow D(\mathcal{A})$ is an equivalence. The \emph{tensor product} $\mathcal{A}\otimes \mathcal{B}$ of two dg categories is defined as follows: the set of objects is the cartesian product of the sets of objects in $\mathcal{A}$ and $\mathcal{B}$ and $\mathrm{Hom}_{\mathcal{A}\otimes \mathcal{B}}((x,w),(y,z)):=\mathrm{Hom}_{\mathcal{A}}(x,y)\otimes\mathrm{Hom}_{\mathcal{B}}(w,z)$ (see \cite{KELS}). Given two dg categories $\mathcal{A}$ and $\mathcal{B}$, let $\mathrm{rep}(\mathcal{A},\mathcal{B})$ be the full triangulated subcategory of $D(\mathcal{A}^{op}\otimes \mathcal{B})$ consisting of those $\mathcal{A}-\mathcal{B}$-bimodules $M$ such that $M(x,-)$ is a compact object of $D(\mathcal{B})$ for every object $x\in \mathcal{A}$. %Now there is a additive symmetric monoidal category $\mathrm{Hmo}_0$ with objects being small dg categories and morphisms being
%\begin{eqnarray*}
%	\mathrm{Hom}_{\mathrm{Hmo}_0}(\mathcal{A},\mathcal{B})\simeq K_0(\mathrm{rep}(\mathcal{A},\mathcal{B})).
%\end{eqnarray*} 
The category $\textbf{dgcat}$ of all (small) dg categories and dg functors carries a Quillen model structure whose weak equivalences are Morita equivalences. Let us denote by $\mathrm{Hmo}$ the homotopy category hence obtained and by $\mathrm{Hmo}_0$ its additivization. Now to any small dg category $\mathcal{A}$ one can associate functorially its \emph{noncommutative motive} $U(\mathcal{A})$ which takes values in $\mathrm{Hmo}_0$. This functor $U\colon \textbf{dgcat}\rightarrow \mathrm{Hmo}_0$ is proved to be the \emph{universal additive invariant} (see \cite{TA1S}). An additive invariant is any functor $E\colon \textbf{dgcat}$ $\rightarrow \mathcal{D}$ taking values in an additive category $\mathcal{D}$ such that
%To any such small dg category $\mathcal{A}$ one can associate functorially its noncommutative motive $U(\mathcal{A})$ which takes values in $\mathrm{Hmo}_0$. This functor $U\colon \textbf{dgcat}\rightarrow \mathrm{Hmo}_0$ is proved to be the \emph{universal additive invariant} (see \cite{TTS}). Recall from \cite{TA1S} that an additive invariant is any functor $E\colon \textbf{dgcat}$ $\rightarrow \mathcal{D}$ taking values in an additive category $\mathcal{D}$ such that
\begin{itemize}
	\item[(\textbf{i})] it sends derived Morita equivalences to isomorphisms,\\
	
	\item[(\textbf{ii})] for any pre-triangulated dg category $\mathcal{A}$ admitting full pre-triangulated dg subcategories $\mathcal{B}$ and $\mathcal{C}$ such that $H^0(\mathcal{A})=\langle H^0(\mathcal{B}), H^0(\mathcal{C})\rangle$ is a semiorthogonal decomposition, the morphism $E(\mathcal{B})\oplus E(\mathcal{C})\rightarrow E(\mathcal{A})$ induced by the inclusions is an isomorphism.
\end{itemize}
The category $\mathrm{NChow}(k)$ of \emph{noncommutative motives} is the pseudo-Abelian envelope of the full subcategory of $\mathrm{Hmo}_0(k)$ consisting of smooth and proper dg categories (see \cite{TS} for details). Recall that every additive invariant $E\colon \textbf{dgcat}$ $\rightarrow \mathcal{D}$ factors through $U\colon \textbf{dgcat}\rightarrow \mathrm{Hmo}_0$. The \emph{noncommutative Chow motive} of a smooth projective $k$-scheme $X$ is defined to be $U(\mathrm{perf}(X))$. In the present paper we will apply the three theorems stated below. We use the following notation: Let $G$ split simply connected semi-simple algebraic group over the field $k$ and $P$ a parabolic subgroup. Recall from \cite{PAS}, Theorem 2.10 that there exits a finite free $\mathrm{Ch}$-homogeneous basis of $R(\widetilde{P})$ over $R(\widetilde{G})$. Moreover, associated to a 1-cocycle $\gamma\colon \mathrm{Gal}(k^s|k)\rightarrow G(k^s)$ and each character $\chi\in \mathrm{Ch}$ one has the Tits' algebras $A_{\chi,\gamma}$ (see \cite{PAS}, 3.1 or \cite{KNUS}, p.377).
%one has the Tit's map (see \cite{PA}, 3.1 or \cite{KNU}, p.377) $\beta_{\gamma}\colon \mathrm{Ch}\rightarrow \mathrm{Br}(k)$ which is a group homomorphism and assigns to each character $\chi\in \mathrm{Ch}$ a central simple algebra $A_{\chi,\gamma}\in \mathrm{Br}(k)$. 
If $\rho_1,...,\rho_n$ is the $\mathrm{Ch}$-homogeneous $R(\widetilde{G})$ basis of $R(\widetilde{P})$ we write $\chi(i)$ for the character such that $\rho_i\in R^{\chi(i)}(\widetilde{P})$. Under this notation one has the following theorem:
\begin{thm}\cite[Theorem 2.1 (i)]{TA1S}
	Let $G$, $P$ and $\gamma$ be as above and $E\colon dgcat\rightarrow D$ an additive invariant. Then every $\mathrm{Ch}$-homogeneous basis $\rho_1,...,\rho_n$ of $R(\widetilde{P})$ over $R(\widetilde{G})$ give rise to an isomorphism
	\begin{eqnarray*}
		\bigoplus^n_{i=1}E(A_{\chi(i),\gamma})\stackrel{\sim}\longrightarrow E({_\gamma}X), 
	\end{eqnarray*}
	where $A_{\chi(i),\gamma}$ stands for the Tits' central simple algebras associated to $\rho_i$. %via $\beta_{\gamma}\colon \mathrm{Ch}\rightarrow \mathrm{Br}(k)$.
\end{thm}
\begin{thm}\cite[Theorem 3.3]{TA1S}
	Let $G$, $P$ and $\gamma$ as in Theorem 4.1. Then $\bigoplus^n_{i=1}U(k)\simeq U({_\gamma}X)$ if and only if the Brauer classes $[A_{\chi(i),\gamma}]$ are trivial.
\end{thm}
For central simple $k$-algebras one has the following comparison theorem:
\begin{thm}\cite[Theorem 2.19]{TAS}
	Let $A_1,...,A_n$ and $B_1,...,B_m$ be central simple $k$-algebras, then the following are equivalent:
	\begin{itemize}
		\item[(\textbf{i})] There is an isomorphism 
		\begin{eqnarray*}
			\bigoplus^n_{i=1}U(A_i)\simeq \bigoplus^m_{j=1}U(B_j).
		\end{eqnarray*}
		
		\item[(\textbf{ii})] The equality $n=m$ holds and for all $1\leq i\leq n$ and all $p$ 
		\begin{eqnarray*}
			[B^p_i]=[A^p_{\sigma_p(i)}]\in \mathrm{Br}(k)
		\end{eqnarray*}
		for some permutations $\sigma_p$ depending on $p$.
	\end{itemize}
\end{thm}
For the proofs of our main results we also need Theorem 4.4 below. Let $\mathrm{CSep}(k)$ be the full subcategory of $\mathrm{NChow}(k)$ consisting of objects of the form $U(A)$ with $A$ a commutative separable $k$-algebra. Analogously, $\mathrm{Sep}(k)$ denotes the full subcategory of $\mathrm{NChow}(k)$ consisting of objects $U(A)$ with $A$ a separable $k$-algebra. And finally, we write $\mathrm{CSA}(k)$ for the full subcategory of $\mathrm{Sep}(k)$ consisting of $U(A)$ with $A$ being a central simple $k$-algebra. Moreover, $\mathrm{CSA}(k)^{\oplus}$ denotes the closure of $\mathrm{CSA}(k)$ under finite direct sums.
\begin{thm}\cite[Corollary 2.13]{TAS}
	There is an equivalence of categories
	\begin{eqnarray*}
		\{U(k)^{\oplus n}\mid n\geq 0\}\simeq\mathrm{CSA}(k)^{\oplus}\times_{\mathrm{Sep}(k)} \mathrm{CSep}(k)
	\end{eqnarray*}
	i.e. $\{U(k)^{\oplus n}\mid n\geq 0\}$ is a $2$-pullback of categories with respect to the inclusion morphisms.
\end{thm} 

\begin{prop}
	Let $X$ and $Y$ be smooth projective varieties over a field $k$. Assume that $X$ and $Y$ admit full generalized exceptional collections $\mathcal{E}_0,...,\mathcal{E}_r$ respectively $\mathcal{F}_0,...,\mathcal{F}_s$ such that $\mathrm{End}(\mathcal{E}_i)\simeq A_i$ and $\mathrm{End}(\mathcal{F}_j)\simeq B_j$ are central simple algebras. Then the ordered collection $\{\mathcal{E}_i\boxtimes \mathcal{F}_j\}$ is a full generalized exceptional collection, where $\mathcal{E}_{i_1}\boxtimes\mathcal{F}_{j_1}$ precedes $\mathcal{E}_{i_2}\boxtimes\mathcal{F}_{j_2}$ iff $(i_1,j_1)\prec (i_2,j_2)$. Here $\prec$ stands for the lexicographic order on $\{0,...,r\}\times\{0,...,s\}$. Moreover, $\mathrm{End}(\mathcal{E}_i\boxtimes \mathcal{F}_j)\simeq A_i\otimes B_j$ and $\mathrm{rdim}(X\times Y)=0$ if and only if $A_i\otimes B_j\simeq M_{n(i,j)}(k)$.
\end{prop}
\begin{proof}
	That the ordered collection $\{\mathcal{E}_i\boxtimes \mathcal{F}_j\}$ is a full generalized exceptional collection follows from \cite{KUZS} and the fact that the tensor product of central simple algebras is again central simple. To show that $\mathrm{End}(\mathcal{E}_i\boxtimes \mathcal{F}_j)\simeq A_i\otimes B_j$ one uses K\"unneth formula.
	Now assume $\mathrm{rdim}(X\times Y)=0$. Let us denote the central simple algebra $A_i\otimes B_j$ by $C_{i,j}$. Then 
	\begin{eqnarray*}
		U(\mathrm{perf}_{dg}(X\times Y))\simeq \bigoplus_{(i,j)}U(C_{i,j}). 
	\end{eqnarray*}
	From \cite{ABS}, Lemma 1.20 we obtain that there is a semiorthogonal decomposition 
	\begin{eqnarray}
	D^b(X\times Y)=\langle \mathcal{A}_1,...,\mathcal{A}_n\rangle
	\end{eqnarray}
	with $\mathcal{A}_i\simeq D^b(K_i)$ and $K_i$ being \'etale $k$-algebras. Therefore
	\begin{eqnarray*}
		U(\mathrm{perf}_{dg}(X\times Y))\simeq \bigoplus_{(i,j)}U(C_{i,j})\simeq U(K_1)\oplus\cdots \oplus U(K_n).
	\end{eqnarray*}
	Using Theorem 4.4 and the universal property of fibre products, the above isomorphism gives rise to an isomorphism 
	\begin{eqnarray*}
		U(\mathrm{perf}_{dg}(X\times Y))\simeq \bigoplus_{(i,j)}U(C_{i,j})\simeq U(k)^{\oplus n}.
	\end{eqnarray*}
	Now Theorem 4.3 implies that $A_i\otimes B_j=C_{i,j}$ must be split. On the other hand if $A_i\otimes B_j=C_{i,j}$ is split, the full generalized exceptional collection $\{\mathcal{E}_i\boxtimes \mathcal{F}_j\}$ gives rise to a semiorthogonal decomposition
	\begin{eqnarray*}
		D^b(X\times Y)=\langle D^b(M_{n(i,j)}(k))\rangle_{(i,j)\in \{0,...,r\}\times\{0,...,s\}}
	\end{eqnarray*}
	and therefore $\mathrm{rdim}(X\times Y)=0$, according to \cite{ABS}, Lemma 1.20. This completes the proof.
\end{proof}
\begin{rema}
	\textnormal{If we assume in Proposition 4.5 that $\mathrm{End}(\mathcal{E}_0)\simeq \mathrm{End}(\mathcal{F}_0)\simeq k$, then $\mathrm{rdim}(X\times Y)=0$ if and only if both $A_i\simeq M_{n_i}(k)$ and $B_j\simeq M_{n_j}(k)$ split. This follows easily from the proof.}
	%\end{rema}
	%\begin{rema}
	
\end{rema}
\section{Non-existence of full exceptional collections}
In this section we show that a finite product of non-split (generalized) Brauer--Severi varieties can not have full exceptional collections. The same is true for finite products of certain twisted forms of smooth quadrics. The idea for the proofs is essentially contained in \cite{RAES}. Nevertheless, we give the proofs, adding to the literature. 

\begin{prop}\label{BR1}
	Let $X$ be a finite product of Brauer--Severi varieties. Then $X$ admits a full exceptional collection if and only if $X$ splits as the product of projective spaces.
\end{prop}
\begin{proof}
	A projective space admits a full exceptional collection according to Example 3.4. Now it is a general fact that the finite product of varieties admitting full exceptional collections have full exceptional collections, too. In fact, this follows from the K\"unneth formula and from general results on generating objects in the derived category of coherent sheaves. Note that this also follows from more general results proved in \cite{KUZS}.
	
	To prove the other implication, we use induction on the number of factors in the finite product. We restrict ourselves to prove the statement only for the case $X=Y_1\times Y_2$. Let us denote by $A_1$ and $A_2$ the central simple algebras corresponding to $Y_1$ respectively $Y_2$. Now assume that $X=Y_1\times Y_2$ admits a full exceptional collection. According to \cite{ABS}, Lemma 1.20 we get $\mathrm{rdim}(X)=0$. The assertion then follows from Proposition 4.5 and Remark 4.6, since a Brauer--Severi variety has a full generalized exceptional collection (see \cite{BERS}) which fulfills the conditions of Proposition 4.5 and Remark 4.6. In particular $A_1$ and $A_2$ are split and hence $X\simeq \mathbb{P}^r_k\times \mathbb{P}^s_k$. 
\end{proof}

\begin{rema}
	\textnormal{The above proposition implies that a product of Brauer--Severi varieties is rational over $k$ if and only if it admits a full exceptional collection. For further varieties satisfying such a condition see for instance \cite{ABS} and \cite{VS} and references therein.}
\end{rema}
\begin{prop}
	Let $X$ be a finite product of generalized Brauer--Severi varieties over a field of characteristic zero. Then $X$ admits a full exceptional collection if and only if $X$ splits as the finite product of Grassmannians.
\end{prop}
\begin{proof}
	A Grassmannian over a field $k$ of characteristic zero admits a full exceptional collection according to \cite{KAS}. One can argue as in Proposition \ref{BR1} to conclude that the finite product of such Grassmannians has a full exceptional collection, too.
	
	In order to prove the other implication, note that from \cite{BLUS} it follows that a generalized Brauer--Severi variety $\mathrm{BS}(d,A)$ associated to a central simple algebra $A$ of degree $n$ over a field $k$ of characteristic zero admits a full generalized exceptional collection. To be precise, denote by $P$ the set of partitions $\lambda=(\lambda_1,...,\lambda_d)$ with $0\leq \lambda_d\leq...\leq \lambda_1\leq n-d$. One can choose a total order $\prec$ on $P$ such that $\lambda\prec \mu$ means that the Young diagram of $\lambda$ is not contained in that of $\mu$. Under this notation, in \textit{loc.cit}. it is proved that there is a full semi-exceptional collection (in the sense of Remark 3.3) $\{...\mathcal{V}_{\lambda},...,\mathcal{V}_{\mu}...\}$ with $\lambda\prec \mu$ and $\mathrm{End}(\mathcal{V}_{\lambda})$ being isomorphic to $A^{\otimes |\lambda|}$. Here $|\lambda|=\lambda_1+...+\lambda_d$, where $\lambda=(\lambda_1,...,\lambda_d)\in P$. By the Wedderburn Theorem $A^{\otimes |\lambda|}\simeq M_{n_{\lambda}}(D_{\lambda})$ with some unique central division algebras $D_{\lambda}$. In particular, this implies that the Krull--Schmidt decomposition of such a $\mathcal{V}_{\lambda}$ is given by $\mathcal{V}_{\lambda}\simeq \mathcal{W}^{\oplus n_{\lambda} }_{\lambda}$. By construction, the ordered set $\{...\mathcal{W}_{\lambda},...,\mathcal{W}_{\mu}...\}$ is a full generalized exceptional collection with $\mathrm{End}(\mathcal{W}_{\lambda})$ being isomorphic to $D_{\lambda}$ and therefore Brauer-equivalent to  $A^{\otimes |\lambda|}$. 
	As in the proof of Proposition 5.1 we restrict to the case when $X=\mathrm{BS}(d_1, A_1)\times\mathrm{BS}(d_2,A_2)$ for two central simple $k$-algebras $A_1$ and $A_2$. Now assume that $X$ admits a full exceptional collection. According to \cite{ABS}, Lemma 1.20 we get $\mathrm{rdim}(X)=0$. The assertion then follows from Proposition 4.5 and Remark 4.6. In particular $A_1$ and $A_2$ are split and hence $X$ is the product of two Grassmannians.  
\end{proof}
%and imitate the proof of Proposition 2.5.1 to conclude that for the noncommutative motive of $\mathrm{perf}_{dg}(X)$ we must have  
%\begin{eqnarray*}
%	U(\mathrm{perf}_{dg}(X))\simeq \bigoplus_{\lambda\in P}U(D_{\lambda})\simeq U(k)^{\oplus d}, 
%\end{eqnarray*}
%Than Theorem 2.4.3 implies that $D_{\lambda}$ is Brauer-equivalent to $k$ for all partitions $\lambda\in P$. In particular $A_1$ and $A_2$ are split and hence $X$ is the product of two Grassmannians.  
Now let $G=\mathrm{PSO}_n$ be over $k$ with $n$ even. Given a 1-cocycle $\gamma\colon \mathrm{Gal}(k^s|k)\rightarrow \mathrm{PSO}_n(k^s)$ we get a twisted form of a quadric ${_\gamma}Q$ which is associated to a central simple $k$-algebra $(A,\sigma)$ of degree $n$ with involution of orthogonal type. Note that ${_\gamma}Q$ is isomorphic to the involution variety $\mathrm{IV}(A,\sigma)$ from Section 2. For any splitting field $L$ of $A$, the variety ${_\gamma}Q\otimes_k L$ is isomorphic to a smooth quadric in $\mathbb{P}^{n-1}_L$. 
\begin{prop}
	Let ${_\gamma}Q$ a twisted form as above and assume the associated central simple algebra $(A,\sigma)$ has trivial discriminant $\delta(A,\sigma)$. Then $\mathrm{rdim}({_\gamma}Q)=0$ if and only if ${_\gamma}Q$ splits, i.e. if $(A,\sigma)$ splits and ${_\gamma}Q$ is a smooth quadric in $\mathbb{P}^{n-1}_{k}$. In particular, ${_\gamma}Q$ admits a full exceptional collection if and only if it splits.
\end{prop}
\begin{proof}
	If $(A,\sigma)$ is split, i.e. if ${_\gamma}Q$ is a smooth quadric in $\mathbb{P}^{n-1}_k$, the existence of a full exceptional collection follows from \cite{KA2S} (also, see the proof of Proposition 7.2 in \cite{BLUS}). Hence $\mathrm{rdim}({_\gamma}Q)=0$ according to \cite{ABS}, Lemma 1.20. Now assume $\mathrm{rdim}({_\gamma}Q)=0$. Again from \cite{ABS}, Lemma 1.20 we obtain that there is a semiorthogonal decomposition 
	\begin{eqnarray}
	D^b({_\gamma}Q)=\langle \mathcal{A}_1,...,\mathcal{A}_r\rangle
	\end{eqnarray}
	with $\mathcal{A}_i\simeq D^b(K_i)$ and $K_i$ being \'etale $k$-algebras. Therefore, the noncommutative motive of $\mathrm{perf}_{dg}({_\gamma}Q))$ decomposes as 
	\begin{eqnarray*}
		U(\mathrm{perf}_{dg}({_\gamma}Q))=U(K_1)\oplus...\oplus U(K_r).
	\end{eqnarray*}
	On the other hand (see \cite{TA1S}, Example 3.11), one has
	\begin{eqnarray*}
		U(\mathrm{perf}_{dg}({_\gamma}Q))\simeq \left(\bigoplus^{n-3}_{\substack{i\geq 0\\even}}U(k)\right)\oplus \left(\bigoplus^{n-3}_{\substack{i>0\\odd}}U(A)\right)\oplus U(C^{+}_0(A,\sigma))\oplus U(C^{-}_0(A,\sigma)).
	\end{eqnarray*}
	Therefore, we have the following isomorphism
	\begin{eqnarray*}
		\left(\bigoplus^{n-3}_{\substack{i\geq 0\\even}}U(k)\right)\oplus \left(\bigoplus^{n-3}_{\substack{i>0\\odd}}U(A)\right)\oplus U(C^{+}_0(A,\sigma))\oplus U(C^{-}_0(A,\sigma))\simeq U(K_1)\oplus...\oplus U(K_r).
	\end{eqnarray*}
	Since the discriminant $\delta(A,\sigma)$ is trivial, the even Clifford algebra $C_0(A,\sigma)$ decomposes as $C_0(A,\sigma)\simeq C^{+}_0(A,\sigma)\times C^{-}(A,\sigma)$, where $C^{+}_0(A,\sigma)$ and $C^{-}_0(A,\sigma)$ are central simple $k$-algebras (see \cite{KNUS}, Theorem 8.10). 
	Using Theorem 4.4 and the universal property of fibre products, the above isomorphism gives rise to an isomorphism 
	\begin{eqnarray*}
		\left(\bigoplus^{n-3}_{\substack{i\geq 0\\even}}U(k)\right)\oplus \left(\bigoplus^{n-3}_{\substack{i>0\\odd}}U(A)\right)\oplus U(C^{+}_0(A,\sigma))\oplus U(C^{-}_0(A,\sigma))\simeq U(k)^{\oplus r}.
	\end{eqnarray*}
	Now Theorem 4.3 implies that $A$ must be split. The same argument shows ${_\gamma}Q$ admits a full exceptional collection if and only if it splits. This completes the proof.  
\end{proof}
Note that in Proposition 5.4 the central simple algebras $k,A,C^{+}_0(A,\sigma)$ and $C^{-}_0(A,\sigma)$ are the minimal Tits' algebras of ${_\gamma}Q$. 
Recall that the Tits' algebras of the product of adjoint semi-simple algebraic groups are Brauer equivalent to the tensor product of the Tits' algebras of its factors (\cite{MPWS}, Corollary 2.3). Using Theorems 4.1, 4.2 and 4.4 and the fact that finite products of the considered varieties have full exceptional collections too, one can imitate the proofs of Proposition 5.1 and 5.3 to show easily: 
\begin{cor}
	Let $X={_{\gamma_1}}Q_1\times\cdots\times{_{\gamma_m}}Q_m$ be the finite product of twisted forms of quadrics as in Proposition 5.4. Let $(A_i,\sigma_i)$ be the central simple algebra with involution of orthogonal type associated to the factor ${_{\gamma_i}}Q_i$. Then $\mathrm{rdim}(X)=0$ if and only if $X$ splits as the product of smooth quadrics, i.e. if all $(A_i,\sigma_i)$ split. In particular, $X$ admits a full exceptional collection if and only if $X$ splits as the product of smooth quadrics.
\end{cor}
%\begin{proof}
%The proof is analogous to the proof of Proposition 5.1.
%\end{proof}
\begin{thm}
	Let $G=\mathrm{PSO}_n$ be over $\mathbb{R}$ with $n$ even and let ${_\gamma}Q$ be a twisted form of a quadric associated to a central simple $\mathbb{R}$-algebra $(A,\sigma)$ of degree $n$ with involution of orthogonal type associated to a non-trivial 1-cocycle $\gamma$.%Given a non-trivial 1-cocycle $\gamma\colon \mathrm{Gal}(\mathbb{C}|\mathbb{R})\rightarrow \mathrm{PSO}_n(\mathbb{C})$, we get a twisted form of a quadric ${_\gamma}Q$ associated to a central simple $\mathbb{R}$-algebra $(A,\sigma)$ of degree $n$ with involution of orthogonal type associated to $\gamma$.  
	Then $\mathrm{rdim}({_\gamma}Q)=0$ if and only if ${_\gamma}Q$ splits, i.e. if $(A,\sigma)$ splits and ${_\gamma}Q$ is a smooth quadric in $\mathbb{P}^{n-1}_{k}$. In particular, ${_\gamma}Q$ admits a full exceptional collection if and only if it splits.
\end{thm}
\begin{proof}
	If $(A,\sigma)$ splits, i.e. if ${_\gamma}Q$ is a smooth quadric in $\mathbb{P}^{n-1}_{\mathbb{R}}$, the existence of a full exceptional collection follows from \cite{KA2S} (see also \cite{BLUS}). Hence $\mathrm{rdim}({_\gamma}Q)=0$ according to \cite{ABS}, Lemma 1.20. Now assume $\mathrm{rdim}({_\gamma}Q)=0$. Then the derived category $D^b({_\gamma}Q)$ must have a semiorthogonal decomposition of the form
	\begin{eqnarray}
	D^b({_\gamma}Q)=\langle \mathcal{A}_1,...,\mathcal{A}_e\rangle
	\end{eqnarray}
	with $\mathcal{A}_i\simeq D^b(\mathbb{R},K_i)$ and $K_i$ being \'etale $\mathbb{R}$-algebras (see \cite{AB1S}, Proposition 6.1.6). We notice that $K_i\simeq \mathbb{R}^{\times n_i}\times\mathbb{C}^{\times m_i}$. Now \cite{ABS}, Lemma 1.17 implies 
	\begin{eqnarray}
	D^b(K_i)\simeq D^b(\mathbb{R})^{\times n_i}\times D^b(\mathbb{C})^{\times m_i}. 
	\end{eqnarray}
	According to \cite{BLUS}, there is a semiorthogonal decomposition with exactly $n-1$ components
	\begin{eqnarray*}
		D^b({_\gamma}Q)=\langle D^b(k), D^b(A),...,D^b(k), D^b(A), D^b(C(A,\sigma))\rangle.
	\end{eqnarray*}
	Note that this semiorthogonal decomposition is induced by vector bundles $\mathcal{V}_1, \mathcal{V}_2,...,\mathcal{V}_{n-1}$ on ${_\gamma}Q$ satisfying $\mathrm{End}(\mathcal{V}_1)=k, \mathrm{End}(\mathcal{V}_2)=A,...,\mathrm{End}(\mathcal{V}_{n-1})=C(A,\sigma)$. Now the noncommutative motive of $\mathrm{perf}_{dg}({_\gamma}Q)$ decomposes as 
	\begin{eqnarray*}
		U(\mathrm{perf}_{dg}({_\gamma}Q))=U(k)\oplus U(A)\oplus...\oplus U(k)\oplus U(A)\oplus U(C(A,\sigma)).
	\end{eqnarray*}
	From \cite{KNUS}, Theorem 8.10 we know that $C(A,\sigma)$ is either a central simple algebra over $\mathbb{C}$ or that $C(A,\sigma)$ splits as the direct product of two central simple $\mathbb{R}$-algebras. In the first case $C(A,\sigma)\simeq M_s(\mathbb{C})$ whereas in the latter case $C(A,\sigma)$ is isomorphic to $A_1\times A_2$, where $A_i$ is isomorphic to $M_{n_1}(\mathbb{R})$ or $M_{n_2}(\mathbb{H})$. By Morita equivalence we have $D^b(M_s(\mathbb{C}))\simeq D^b(\mathbb{C})$, $D^b(M_{n_1}(\mathbb{R}))\simeq D^b(\mathbb{R})$ and $D^b(M_{n_2}(\mathbb{H}))\simeq D^b(\mathbb{H})$. So there are two cases to consider.\\
	\noindent 
	First case:
	If $C(A,\sigma)\simeq M_s(\mathbb{C})$, one has $\mathrm{rk}K_0({_\gamma}Q)=n-1$ and from (3) and (4) we conclude
	\begin{eqnarray}
	n-1=\sum^e_{i=1} n_i+\sum^e_{i=1} m_i.
	\end{eqnarray}
	After base change to $\mathbb{C}$, the vector bundles $\mathcal{V}_1\otimes_{\mathbb{R}}\mathbb{C}, \mathcal{V}_2\otimes_{\mathbb{R}}\mathbb{C},...,\mathcal{V}_{n-1}\otimes_{\mathbb{R}}\mathbb{C}$ on the smooth quadric ${_\gamma}Q\otimes_{\mathbb{R}}\mathbb{C}$ give rise to a semiorthogonal decomposition of $D^b({_\gamma}Q\otimes_{\mathbb{R}}\mathbb{C})$. Moreover, we have 
	\begin{eqnarray*}
		\mathrm{End}(\mathcal{V}_{n-1}\otimes_{\mathbb{R}}\mathbb{C}))\simeq\mathrm{End}(\mathcal{V}_{n-1})\otimes_{\mathbb{R}}\mathbb{C}\simeq M_s(\mathbb{C})\otimes_{\mathbb{R}}\mathbb{C}\simeq M_s(\mathbb{C})\times M_s(\mathbb{C})
	\end{eqnarray*}
	and therefore $\mathrm{rk}K_0({_\gamma}Q\otimes_{\mathbb{R}}\mathbb{C})=n$. 
	Since $\mathrm{End}(\mathcal{V}_{n-1})\otimes_{\mathbb{R}}\mathbb{C} \simeq M_s(\mathbb{C})\times M_s(\mathbb{C})$, we obtain with the semiorthogonal decomposition (3) in view of (4) that 
	\begin{eqnarray*}
		n=\sum^e_{i=1} n_i+2\sum^e_{i=1} m_i.
	\end{eqnarray*}
	Taking the difference of the latter equation with (5) we find
	\begin{eqnarray}
	1=\sum^e_{i=1}m_i
	%n=\sum^e_{i=1} n_i+2.
	\end{eqnarray}
	%From (6) and (7) we find $1=\sum^e_{i=1}m_i$. Without loss of generality, we can assume $m_1=...=m_{e-1}=0, m_e=1$. 
	Without loss of generality, we can assume $m_1=...=m_{e-1}=0, m_e=1$.
	This gives us the following isomorphism of noncommutative motives
	\begin{eqnarray*}
		U(\mathbb{R})^{\oplus (n-2)}\oplus U(\mathbb{C})\simeq U(\mathbb{R})\oplus U(A)\oplus\cdots\oplus U(\mathbb{R})\oplus U(A)\oplus U(\mathbb{C}).
	\end{eqnarray*}
	Now \cite{TA2S}, Proposition 4.5 implies 
	\begin{eqnarray*}
		U(\mathbb{R})^{\oplus (n-2)}\simeq U(\mathbb{R})\oplus U(A)\oplus\cdots\oplus U(\mathbb{R})\oplus U(A)
	\end{eqnarray*}
	and Theorem 4.3 shows that $A$ splits. This gives the assertion for the case $C(A,\sigma)\simeq M_s(\mathbb{C})$.\\
	\noindent
	Second case: Now let $C(A,\sigma)=A_1\times A_2$ be the product of two central simple $\mathbb{R}$-algebras. In this case the rank of the Grothendieck group $K_0({_\gamma}Q)$ must be $n$. Therefore 
	\begin{eqnarray}
	n=\sum^e_{i=1} n_i+\sum^e_{i=1} m_i.
	\end{eqnarray}
	After base change to $\mathbb{C}$, we obtain
	\begin{eqnarray}
	n=\sum^e_{i=1} n_i+2\sum^e_{i=1} m_i.
	\end{eqnarray}
	From (7) and (8) we find $0=\sum^e_{i=1} m_i$. This gives 
	\begin{eqnarray*}
		U(\mathbb{R})^{\oplus n}\simeq U(\mathbb{R})\oplus U(A)\oplus\cdots\oplus U(\mathbb{R})\oplus U(A)\oplus U(A_1)\oplus U(A_2).
	\end{eqnarray*}
	Again Theorem 4.3 implies that $A$ splits. This completes the proof.
\end{proof}
\noindent
The next result generalizes Theorem 5.6. We give the proof in detail, although it is similar to the proof of Theorem 5.6.
\begin{thm}
	Let $X={_{\gamma_1}}Q_1\times\cdots\times{_{\gamma_m}}Q_m$ be a finite product of twisted quadrics as in Theorem 5.6. Then $\mathrm{rdim}(X)=0$ if and only if $X$ splits as the product of smooth quadrics, i.e. if all $(A_i,\sigma_i)$ split. In particular, $X$ admits a full exceptional collection if and only if $X$ splits as the product of smooth quadrics. 
\end{thm}
\begin{proof}
	We give a proof only for the case $X={_{\gamma_1}}Q_1\times{_{\gamma_2}}Q_2$. Let $(A_1,\sigma_1)$ and $(A_2,\sigma_2)$ be the involution algebras associated with ${_{\gamma_1}}Q_1$ and ${_{\gamma_2}}Q_2$. Denote by $n$ the degree of $(A_1,\sigma_1)$ and by $m$ the degree of $(A_2,\sigma_2)$. We restrict ourselves to the case where $C(A_1,\sigma_1)\simeq M_s(\mathbb{C})$ and $C(A_2,\sigma_2)\simeq M_r(\mathbb{C})$, since the other cases are proved analogously. 
	Recall from \cite{BLUS} that ${_{\gamma_1}}Q_1$ and ${_{\gamma_2}}Q_2$ have semiorthogonal decompositions
	\begin{eqnarray}
	D^b({_\gamma}_1Q_1)=\langle D^b(k), D^b(A_1),...,D^b(k), D^b(A_1), D^b(C(A_1,\sigma_1))\rangle
	\end{eqnarray}
	and 
	\begin{eqnarray}
	D^b({_\gamma}_2Q_2)=\langle D^b(k), D^b(A_2),...,D^b(k), D^b(A_2), D^b(C(A_2,\sigma_2))\rangle.
	\end{eqnarray}
	From \cite{BLUS} it follows that the number of components in (9) respectively (10) is $n-1$, respectively $m-1$. According to \cite{KUZS}, the product $X={_{\gamma_1}}Q_1\times{_{\gamma_2}}Q_2$ has a semiorthogonal decomposition which is constructed by taking successive tensor products of the components of the semiorthogonal decompositions (9) and (10). Considering the case $C(A_1,\sigma_1)\simeq M_s(\mathbb{C})$ and $C(A_2,\sigma_2)\simeq M_r(\mathbb{C})$, it is an exercise to verify that 
	\begin{eqnarray}
	\mathrm{rk}K_0(X)=(n-1)(m-1)+1. 
	\end{eqnarray}
	This is also the number of components in the semiorthogonal decomposition of $D^b(X)$, obtained by taking successive tensor products of the components of (9) and (10). The number of components in the semiorthogonal decomposition of $D^b(X)$ that are equivalent to $D^b(\mathbb{C})$ is 
	\begin{eqnarray*}
		(n-1)(m-1)+1-(n-2)(m-2)=m+n-2.  
	\end{eqnarray*}
	So after base change to $\mathbb{C}$, we obtain
	\begin{eqnarray}
	\mathrm{rk}K_0(X\otimes_{\mathbb{R}}\mathbb{C})= (n-2)(m-2)+2(m+n-2). 
	\end{eqnarray}
	Assuming $\mathrm{rdim}(X)=0$, there must be a semiorthogonal decomposition 
	\begin{eqnarray}
	D^b({_\gamma}Q)=\langle \mathcal{A}_1,...,\mathcal{A}_e\rangle
	\end{eqnarray}
	with $\mathcal{A}_i\simeq D^b(K_i)$ and $K_i$ being \'etale $\mathbb{R}$-algebras. As in the proof of Proposition 5.6, we have 
	\begin{eqnarray}
	D^b(K_i)\simeq D^b(\mathbb{R})^{\times n_i}\times D^b(\mathbb{C})^{\times m_i}. 
	\end{eqnarray}
	Comparing (11), (13) and (14), we find
	\begin{eqnarray}
	(n-1)(m-1)+1=\sum^e_{i=1} n_i+\sum^e_{i=1} m_i.
	\end{eqnarray}
	From (12) and base change of the semiorthogonal decomposition (13) to $\mathbb{C}$, we get
	\begin{eqnarray}
	(n-2)(m-2)+2(m+n-2)=nm=\sum^e_{i=1} n_i+2\sum^e_{i=1} m_i.
	\end{eqnarray} 
	Taking the difference of the equation (16) with equation (15) implies $n+m-2=\sum^e_{i=1} m_i$ and therefore $nm-2n-2m+4=\sum^e_{i=1} n_i$. Now (13) gives the following isomorphism for the noncommutative motive of $\mathrm{perf}_{dg}(X)$
	\begin{eqnarray*}
		U(\mathrm{perf}_{dg}(X))\simeq U(\mathbb{R})^{\oplus (nm-2n-2m+4)}\oplus U(\mathbb{C})^{\oplus (n+m-2)}.
	\end{eqnarray*}
	The semiorthogonal decomposition of the product $X={_{\gamma_1}}Q_1\times{_{\gamma_2}}Q_2$ gives
	\begin{eqnarray*}
		U(\mathrm{perf}_{dg}(X))\simeq \big((U(\mathbb{R})\oplus U(A_1)\oplus U(A_2)\oplus U(A_1\otimes A_2)\big)^{\oplus \frac{(n-2)(m-2)}{4}}\oplus U(\mathbb{C})^{\oplus (n+m-2)}.
	\end{eqnarray*}
	Comparing both isomorphisms for $U(\mathrm{perf}_{dg}(X))$, Proposition 4.5 of \cite{TA2S} implies 
	\begin{eqnarray*}
		\big((U(\mathbb{R})\oplus U(A_1)\oplus (A_2)\oplus U(A_1\otimes A_2)\big)^{\oplus \frac{(n-2)(m-2)}{4}}\simeq  U(\mathbb{R})^{\oplus (nm-2n-2m+4)} 
	\end{eqnarray*}
	Theorem 4.3 implies that $A_1$ and $A_2$ split. This completes the proof.
\end{proof}

More generally, Theorem 4.2 has the following consequence.
\begin{prop}
	Let $G$ be a split simply connected simple algebraic group of classical type $A_n$ or $C_n$ over $k$ and $P\subset G$ a parabolic subgroup. Moreover, let $\gamma$ be a 1-cocycle, %\colon \mathrm{Gal}(k^s|k)\rightarrow G(k^s)$ be a 1-cocycle, 
	$X=G/P$ the homogeneous variety and ${_\gamma}X$ its twisted form. If ${_\gamma}X$ admits a full exceptional collection, then $\gamma$ must be the trivial 1-cocycle. In other words, non-trivial twisted flags of the considered type do not have full exceptional collections.
\end{prop}
\begin{proof}
	Let $\widetilde{G}$ and $\widetilde{P}$ be the universal covers of $G$ and $P$ respectively and $R(\widetilde{G})$ and $R(\widetilde{P})$ the associated representation rings. One has the character group $\mathrm{Ch}=\mathrm{Hom}(\widetilde{Z},\mathbb{G}_m)$ which is a finite group. Now let $\rho_1,...,\rho_n$ be a $\mathrm{Ch}$-homogeneous basis of $R(\widetilde{P})$ over $R(\widetilde{G})$ and $A_{\chi(i),\gamma}$ the Tits' central simple algebras associated to $\rho_i$. It is clear that $n\geq \mathrm{ord}(\mathrm{Ch})$ (see \cite{PAS} for details on the basis $\rho_1,...,\rho_n$). Assuming the existence of a full exceptional collection in $D^b({_\gamma}X)$, Theorem 4.1 gives an isomorphism 
	\begin{eqnarray}
	\bigoplus^n_{i=1}U(A_{\chi(i),\gamma})\simeq U({_\gamma}X)\simeq U(k)\oplus...\oplus U(k)
	\end{eqnarray}
	Now Theorem 4.2 implies that the Brauer classes $[A_{\chi(i),\gamma}]$ must be trivial for all characters $\chi(i)$. % Since $n\geq \mathrm{ord}(\mathrm{Ch})$, the Tit's map $\beta_{\gamma}\colon \mathrm{Ch}\rightarrow \mathrm{Br}(k)$ must be trivial. 
	From the classification of the minimal Tits' algebras for simply connected classical groups (see \cite{KNUS}, p.378) we conclude that the 1-cocycle $\gamma$ must be trivial. This completes the proof.
\end{proof}
\begin{rema}
	\textnormal{We believe that in Proposition 5.8 we actually have if and only if.}
\end{rema}
\begin{cor}
	Let $G_1,...,G_n$ be split simply connected simple algebraic groups as in Proposition 5.8 and $\gamma_i$ %\colon \mathrm{Gal}(k^s|k)\rightarrow G_i(k^s)$, $1\leq i\leq n$, 
	1-cocycles. Given parabolic subgroups $P_i\subset G_i$, let ${_{\gamma_i}}(G_i/P_i)$ be the twisted forms of $G_i/P_i$. If $X={_{\gamma_1}}(G_1/P_1)\times...\times {_{\gamma_n}}(G_n/P_n)$ admits a full exceptional collection, then all 1-cocycles $\gamma_i$ must be trivial. 
\end{cor}
\begin{proof}
	We just sketch the proof as it is completely analogous to that of Proposition 5.8. Consider the classification of the minimal Tits' algebras for simply connected classical groups given in \cite{KNUS} (see also \cite{MPWS}). Then the successive tensor products of the Tits' algebras of each factor gives simple algebras that are Brauer equivalent to the Tits' algebras of $X$. From Theorem 4.2 we conclude that all these algebras must be trivial in the respective Brauer group. Hence all 1-cocycles $\gamma_i$ must be trivial. 
\end{proof}

\section{Categorical representability in dimension zero and rational points}
Propositions 5.1, 5.3, 5.4 and Theorems 5.6 and 5.7 are closely related to the existence of $k$-rational points on the considered varieties. For details and further results in this direction see \cite{AB1S}, \cite{ABS}, \cite{BBS} and \cite{MBTS}. Recall the definition of categorical representability from \cite{AB1S}.
\begin{defi}
	\textnormal{A $k$-linear triangulated category $T$ is \emph{representable in dimension $m$} if it admits a semiorthogonal decomposition $T=\langle \mathcal{A}_1,...,\mathcal{A}_r\rangle$ and for each $i=1,...,r$ there exists a smooth projective connected variety $Y_i$ with $\mathrm{dim}(Y_i)\leq m$, such that $\mathcal{A}_i$ is equivalent to an admissible subcategory of $D^b(Y_i)$.}
\end{defi}
\noindent
We will use the following notation:
\begin{eqnarray*}
	\mathrm{rdim}(T):=\mathrm{min}\{m\mid T\  \textnormal{is representable in dimension m}\},
\end{eqnarray*}
whenever such a finite $m$ exists. 
\begin{defi}
	\textnormal{Let $X$ be a smooth projective $k$-variety. One says $X$ is \emph{representable in dimension} $m$ if $D^b(X)$ is representable in dimension $m$. We will write
		\begin{eqnarray*}
			\mathrm{rdim}(X):=\mathrm{rdim}(D^b(X)), 
			\thickspace \mathrm{rcodim}(X):= \mathrm{dim}(X)-\mathrm{rdim}(X). 
	\end{eqnarray*}}
\end{defi}
Based on the idea of the proof of Proposition 5.1, we obtain:
\begin{thm}
	Let $X$ be the finite product of Brauer--Severi varieties over a field $k$. Then the following are equivalent:
	\begin{itemize}
		\item[(\textbf{i})] $\mathrm{rdim}(X)=0$.
		\item[(\textbf{ii})] $X$ admits a full exceptional collection.
		\item[(\textbf{iii})] $X$ is rational over $k$.
		\item[(\textbf{iv})] $X$ admits a $k$-rational point.
	\end{itemize} 
\end{thm}
\begin{proof}
	The equivalence of (iii) and (iv) is an application of a Theorem of Ch\^{a}telet (see \cite{GSS}, Theorem 5.1.3) and the Lang--Nishimura Theorem. To be precise, if $X=Y_1\times\cdots\times Y_r$ is birational over $k$ to some $\mathbb{P}^n_k$, we can consider the rational map $\mathbb{P}^n_k\dashrightarrow X$ to obtain a $k$-rational point on $X$ from the Lang--Nishimura Theorem. Now assume $X$ admits a $k$ rational point. Note that we have the projections $p_i:X\rightarrow Y_i$ which, again by the Lang--Nishimura Theorem, provide us with $k$-rational points on every $Y_i$. The before mentioned Theorem of Ch\^{a}telet implies $Y_i\simeq \mathbb{P}^{m_i}_k$ for all $i\in \{1,...,r\}$. Thus $X\simeq \mathbb{P}^{m_1}_k\times...\times\mathbb{P}^{m_r}_k$. The equivalence of (ii) and (iii) is the content of Remark 5.2. 
	
	It remains to show that (i) is equivalent to (iii). So we assume $\mathrm{rdim}(X)=0$. From Proposition 5.1 it follows that $X$ splits as the direct product of projective spaces. Therefore, $X$ is $k$-rational. On the other hand, if $X$ is rational over $k$ there is a birational map $\mathbb{P}^s_k\dashrightarrow X=Y_1\times Y_2$ and by the Lang--Nishimura Theorem $X$ admits a $k$-rational point. We conclude $X\simeq \mathbb{P}^n_k\times \mathbb{P}^m_k$. Therefore $X$ has a full (strong) exceptional collection. Again, Lemma 1.20 of \cite{ABS} implies $\mathrm{rdim}(X)=0$. This completes the proof.
\end{proof}
\begin{rema}
	\textnormal{In \cite{ABS}, the statement of Theorem 6.3 is proved for the special case where $X$ is a Brauer--Severi surface. The proof in \textit{loc.cit}. however relies on the transitivity of the Braid group action on the set of full exceptional collections on $X\otimes_k k^s=\mathbb{P}^2_{k^s}$ and makes use of Galois descent.}
\end{rema}
\begin{thm}
	Let $X$ be a finite product of generalized Brauer--Severi varieties over a field $k$ of characteristic zero. Then $\mathrm{rdim}(X)=0$ if and only if $X$ splits as the finite product of Grassmannians over $k$. 
\end{thm}
\begin{proof}
	We mentioned in Proposition 5.3 that a finite product $X$ of Grassmannians over $k$ admits a full exceptional collection. Lemma 1.20 of \cite{ABS} immediately implies $\mathrm{rdim}(X)=0$.

	Now assume $\mathrm{rdim}(X)=0$. Now see the proof of Proposition 5.3 to conclude with  Proposition 4.5 that $X$ splits as the direct product of Grassmannians. Therefore, $X$ is $k$-rational. 
\end{proof}
\begin{cor}
	Let $X$ be the finite product of generalized Brauer--Severi varieties over a field of characteristic zero. Then $\mathrm{rdim}(X)=0$ if and only if it admits a full exceptional collection.
\end{cor}
\begin{rema}
	\textnormal{Theorem 6.5 shows that if $\mathrm{rdim}(X)=0$ for a finite product $X$ of generalized Brauer--Severi varieties over a field $k$ of characteristic zero, then $X$ must be rational over $k$ and has therefore a $k$-rational point. Note that the other implication does not hold. Indeed, let $(a,b)$ be a non-split quaternion algebra over a field $k$ of characteristic zero. Consider the central simple algebra $A=M_n((a,b))$ for $n\geq 2$ and the associated generalized Brauer--Severi variety $X=\mathrm{BS}(d,A)$ with for instance $d=4$. As $\mathrm{ind}((a,b))=2$ divides $d=4$, results in \cite{BLS} show that $X$ admits a $k$-rational point. But Theorem 6.15 from below implies $\mathrm{rdim}(X)=1\neq 0$.}
\end{rema}
Recall that a central simple algebra $(A,\sigma)$ with involution is called \emph{isotropic} if $\sigma(a)\cdot a=0$ for some nonzero element $a\in A$. %it contains a non-zero isotropic ideal. A right ideal $I\subset A$ is called \emph{isotropic} if $I\subset I^{\perp}=\{x\in A| \sigma(x)y=0$ for $y\in I\}$. Note that an isotropic ideal remains isotropic after scalar extension. 
So after base change to some splitting field $L$ of $A$ the involution variety associated to an isotropic involution algebra $(A,\sigma)$ becomes a quadric $V(q)$ where $q$ is an isotropic quadratic form over $L$ (see \cite{KNUS}, p.74 Example 6.6). 
\begin{prop}
	Let ${_\gamma}Q$ be as in Proposition 5.4 and assume the associated involution algebra $(A,\sigma)$ is isotropic. Then the following are equivalent:
	\begin{itemize}
		\item[(\textbf{i})]	$\mathrm{rdim}({_\gamma}Q)=0$.
		\item[(\textbf{ii})] ${_\gamma}Q$ admits a full exceptional collection.
		\item[(\textbf{iii})] ${_\gamma}Q$ is rational over $k$.
		\item[(\textbf{iv})] ${_\gamma}Q$ admits a $k$-rational point.
	\end{itemize} 
\end{prop}
\begin{proof}
	If ${_\gamma}Q$ admits a full exceptional collection, $\mathrm{rdim}({_\gamma}Q)=0$. Now if $\mathrm{rdim}({_\gamma}Q)=0$, Proposition 5.4 implies that $A$ is split and therefore ${_\gamma}Q$ must be a smooth quadric induced by an isotropic quadratic form. Hence ${_\gamma}Q$ admits a full exceptional collection. This proves the equivalence of (i) and (ii). Assume (ii). Then Proposition 5.4 shows that $A$ is split and hence ${_\gamma}Q$ is isomorphic to a smooth isotropic quadric which has a $k$-rational point. On the other hand, if ${_\gamma}Q$ admits a $k$-rational point, Proposition 4.3 of \cite{DTS} implies that $A$ is split and so ${_\gamma}Q$ must be a smooth quadric (which admits a full exceptional collection). This proves the equivalence of (ii) and (iii). Now it is well-known that (iii) and (iv) are equivalent (see for instance \cite{COS}).
\end{proof}
\begin{cor}
	Proposition 6.8 also holds for finite products.
\end{cor}
\begin{proof}
	Let $X=X_1\times\cdots\times X_n$ be the finite product of twisted quadrics as in Proposition 5.4. We show the equivalence of (i) and (ii). For this, assume $X$ admits a full exceptional collection. Then $\mathrm{rdim}(X)=0$. On the other hand, if $\mathrm{rdim}(X)=0$, Corollary 5.5 implies that $X$ is isomorphic to the product of smooth isotropic quadrics and hence admits a full exceptional collection. Now assume $X$ admits a full exceptional collection. Then Corollary 5.5 shows that $X$ is isomorphic to the product of smooth isotropic quadrics and so has a $k$-rational point. On the other hand, the existence of a $k$-rational point on $X$ gives us a $k$-rational point on any $X_i$ by the Lang--Nishimura Theorem. Then Proposition 4.3 of \cite{DTS} implies that any $X_i$ must be a smooth isotropic quadric. Since a smooth quadric admits a full exceptional collection, the product $X$ admits a full exceptional collection, too. This shows the equivalence of (ii) and (iii). Finally, if $X$ has a $k$-rational point, each $X_i$ has a $k$-rational point and must be a smooth isotropic quadric, which is $k$-rational. Therefore $X$ is rational. Clearly, if $X$ is $k$-rational it admits a $k$-rational point.
\end{proof}
\begin{prop}
	Let ${_\gamma}Q$ be as in Proposition 5.6 and assume the associated involution algebra $(A,\sigma)$ is isotropic. Then the following are equivalent:
	\begin{itemize}
		\item[(\textbf{i})]	$\mathrm{rdim}({_\gamma}Q)=0$.
		\item[(\textbf{ii})] ${_\gamma}Q$ admits a full exceptional collection.
		\item[(\textbf{iii})] ${_\gamma}Q$ is rational over $\mathbb{R}$.
		\item[(\textbf{iv})] ${_\gamma}Q$ admits a $\mathbb{R}$-rational point.
	\end{itemize} 
\end{prop}
\begin{proof}
	The proof is the same as for Theorem 6.3 with the difference that one uses Theorem 5.6 instead of Proposition 5.1.
	%If ${_\gamma}Q$ admits a full exceptional collection, $\mathrm{rdim}({_\gamma}Q)=0$. Now if $\mathrm{rdim}({_\gamma}Q)=0$, Proposition 5.6 implies that $A$ is split and therefore ${_\gamma}Q$ must be a quadric coming from an isotropic quadratic form. Hence ${_\gamma}Q$ admits a full exceptional collection. This proves the equivalence of (i) and (ii). Assume (ii). Then Proposition 5.6 shows $A$ is split and Proposition 4.3 of \cite{DT} implies that ${_\gamma}Q$ has a $k$-rational point. On the other hand, the existence of a $k$-rational point on ${_\gamma}Q$ implies, again by Proposition 4.3 of \cite{DT}, that $A$ is split and that therefore ${_\gamma}Q$ must be a quadric, which admits a full exceptional collection. This proves the equivalence of (ii) and (iii). Now it is well-known that (ii) and (iv) are equivalent (see for instance \cite{CO}).
\end{proof}
%\end{prop}
\begin{cor}
	Proposition 6.10 also holds for finite products.
\end{cor}
\begin{prop}
	Let $X=X_1\times\cdots\times X_n$ be a finite product over $k$ ($\mathrm{char}(k)\neq 2$) of dimension at least two. Assume $X_i$ is either a Brauer--Severi variety or a twisted form of a quadric as in Proposition 5.4. If $X$ is $k$-rational, then $\mathrm{rcodim}(X)\geq 2$. 
\end{prop}
\begin{proof}
	If $X$ is $k$-rational, it admits a $k$-rational point. By Lang--Nishimura Theorem, any of the factors $X_i$ admits a $k$-rational point. If $X_i$ is a Brauer--Severi variety, Theorem 6.3 implies that $X_i$ has a full exceptional collection. If $X_i$ is a twisted form of a smooth quadric (as in Proposition 5.4), Proposition 6.8 shows that $X_i$ has a full exceptional collection. Therefore, the product $X$ admits a full exceptional collection and hence $\mathrm{rdim}(X)=0$. As $\mathrm{dim}(X)\geq 2$, we conclude $\mathrm{rcodim}(X)\geq 2$.
\end{proof}

\begin{prop}
	Let $X=X_1\times\cdots\times X_n$ be a finite product over $\mathbb{R}$ of dimension at least two. Assume $X_i$ is either a Brauer--Severi variety or a twisted form of a quadric as in Proposition 5.6. If $X$ is $k$-rational, then $\mathrm{rcodim}(X)\geq 2$. 
\end{prop}
%\begin{prop}
%	Let $X=\mathrm{BS}(d_1,A_1)\times\cdots\times \mathrm{BS}(d_n,A_n)$ be a finite product of generalized Brauer--Severi varieties over a field $k$ of chracteristic zero. Assume $\mathrm{ind}(A_i)\leq 3$. If $X$ is $k$-rational, then $\mathrm{rcodim}(X)\geq 2$. 	
%\end{prop}
%\begin{proof}

%\end{proof}

Finally, we want to calculate $\mathrm{rdim}(X)$ for generalized Brauer--Severi and certain involution varieties in the case $A$ is non-split. We first generalize \cite{NO3S}, Proposition 4.1.  
\begin{prop}
	Let $X=\mathrm{BS}(r,A)$ be a generalized Brauer--Severi variety over a field $k$ of characteristic zero. Then $\mathrm{rdim}(X)\leq \mathrm{ind}(A)-1$. In particular $\mathrm{mcd}(X)\leq \mathrm{ind}(A)-1$.	
\end{prop}
\begin{proof}
	Let $a=(a_1,...,a_n)\in \mathbb{N}^r$, subject to the condition $n\geq a_1\geq\cdots \geq a_1\geq 0$ (that is Young diagrams with at most $n-r$ rows and at most $r$ colums). By \cite{BLUS} we have a semiorthogonal decomposition
	\begin{eqnarray}
	D^b(X)=\langle \mathcal{A}_1, \mathcal{A}_2,..., \mathcal{A}_s\rangle,
	\end{eqnarray}
	with components $\mathcal{A}_i$ being equivalent to $D^b(A^{\otimes d(a)})$, where $d(a)=a_1+\cdots +a_n$. According to the Wedderburn Theorem, the central simple algebra $A$ is isomorphic to $M_q(B)$ for a unique central division algebra $B$ and some $q>0$. %Hence $\mathrm{deg}(A)=n\cdot\mathrm{ind}(A)$ and $\mathrm{dim}(X)=\mathrm{deg}(A)-1$. 
	Now let $Y_B$ be the Brauer--Severi variety corresponding to $B$. As $B$ is a division algebra, we have $\mathrm{dim}(Y_B)=\mathrm{ind}(A)-1=\mathrm{ind}(B)-1$. Note that $D^b(A^{\otimes d(a)})\simeq D^b(B^{\otimes d(a)})$. Since $D^b(B^{\otimes i})$ is an admissible subcategory of $D^b(Y_B)$ for any $i\geq 0$ (see \cite{BERS}), we immediately conclude 
	\begin{eqnarray*}
		\mathrm{rdim}(X)\leq \mathrm{dim}(Y_B)=\mathrm{ind}(X)-1.
	\end{eqnarray*}	
\end{proof}
\noindent
The next result generalizes \cite{NO3S}, Theorem 1.4.
\begin{thm}
	Let $X=\mathrm{BS}(r,A)$ be a generalized Brauer--Severi variety over a field $k$ of characteristic zero with $\mathrm{ind}(A)\leq 3$. Then $\mathrm{rdim}(X)=\mathrm{ind}(A)-1$. In particular $\mathrm{mcd}(X)\leq \mathrm{ind}(A)-1$.		
\end{thm}
\begin{proof}
	According to Theorem 6.5 a generalized Brauer--Severi variety $X$ is split if and only if $\mathrm{rdim}(X)=0$. It is a general fact that $X$ splits if and only if $\mathrm{ind}(A)=1$. Therefore,  $\mathrm{rdim}(X)=\mathrm{ind}(A)-1=0$. This covers the case $\mathrm{ind}(A)=1$. So it remains to prove the assertion for $r:=\mathrm{ind}(A)\in\{2,3\}$. 
	%The central simple algebra $A$ is isomorphic to $M_n(B)$ for a unique central division algebra $B$ and some $n>0$. As $r\in\{2,3\}$, the index of $B$ must be $2$ or $3$. %Since $2$ and $3$ are primes and index of $A$ is divided by the period of $A$ (see 2.1), we find $\mathrm{per}(A)=\mathrm{ind}(A)$. Note that $\mathrm{per}(A)=1$ is excluded as $X$ is assumed to be non-split. Since $r\in \{2,3\}$ we have $(\mathrm{per}(A)-1)\in\{1,2\}$. We start with the case $\mathrm{per}(A)-1=1$. 
	By Proposition 6.14 one has $\mathrm{rdim}(X)\leq \mathrm{ind}(X)-1=r-1$. Assume by contradiction that $\mathrm{rdim}(X)< r-1$. So for $r=2$ this means $\mathrm{rdim}(X)=0$ which gives a contradiction, as $A$ is by assumption non-split. For $n=3$, $\mathrm{rdim}(X)< r-1$ means $\mathrm{rdim}(X)\leq 1$. But $\mathrm{rdim}(X)=0$ gives a contradiction, since $X$ is non-split. Below we prove that $\mathrm{rdim}(X)=1$ also gives a contradiction. 
	
	From \cite{AB1S}, Proposition 6.1.6 and 6.1.10 we conclude that if $\mathrm{rdim}(X)=1$ there must be a semiorthogonal decomposition of $D^b(X)$ whose components are either $D^b(K)$, where $K/k$ is a (finite) separable extension, or $D^b(D)$, where $D$ is a central division algebra with $\mathrm{ind}(D)\in\{1,2\}$, or $D^b(C)$, where $C$ is a smooth $k$-curve of positive genus. Note that $D^b(C)$ cannot be present because $K_0(X)$ is torsion free. Using the described semiorthogonal decomposition (that we get assuming $\mathrm{rdim}(X)=1$), we see that the noncommutative motive $U(\mathrm{perf}_{dg}(X))$ decomposes as
	\begin{eqnarray}
	U(\mathrm{perf}_{dg}(X))&\simeq&\bigoplus^n_{i=1}U(K_i)\oplus \left(\bigoplus^m_{j=1}U(D_j)\right)\\
	&\simeq& \bigoplus_a U(A^{\otimes d(a)})
	\end{eqnarray}
	for suitable central division algebras $D_j$ with $\mathrm{ind}(D_j)\in\{1,2\}$ and suitable separable extensions $K_i$.
	Note that $K_0(X)\simeq \mathbb{Z}^{\oplus s}$ and therefore $n+m=s$. After base change to $\bar{k}$ we also have $K_0(X_{\bar{k}})\simeq \mathbb{Z}^{\oplus s}$. %Let us denote the base change of the components $\mathcal{A}_i$ by $\bar{\mathcal{A}}_i$. 
	Now for any $D^b(K_i)$ we obtain after base change $D^b(K_i)_{\bar{k}}\simeq \prod^{r_i}_{q=1}D^b(\bar{k})$.
	%The base change of (2) to $k^s$ gives for the components $\mathcal{A}_i$ for which $\mathcal{A}_i\simeq D^b(K_i)$, $\bar{\mathcal{A}}_i\simeq \prod^{r_i}_{q=1}D^b(k^s)$.
	Hence $\sum^n_{i=1}{r_i}+m=s$. But this implies $\sum^n_{i=1}{r_i}=n$ and therefore $r_i=1$. The above isomorphisms (19) and (20) then give
	\begin{eqnarray*}
		\bigoplus^n_{i=1}U(k)\oplus \left(\bigoplus^m_{j=1}U(D_j)\right) \simeq \bigoplus_a U(A^{\otimes d(a)})
	\end{eqnarray*} 
	Then by Theorem 4.3 we conclude that $A$ must be split or that $A$ must be Brauer-equivalent to $D_j$ for some $j$. This contradicts $r=3$ and completes the proof.
\end{proof}
\noindent
For involution varieties, we have:
\begin{thm}
	Let $X$ be one of the following varieties:
	\begin{itemize}
		\item[(\textbf{i})] a twisted quadric associated to a central simple algebra $(A,\sigma)$ with involution of orthogonal type having trivial discriminant $\delta(A,\sigma)$. 
		\item[(\textbf{ii})] a twisted quadric associated to a central simple $\mathbb{R}$-algebra $(A,\sigma)$ with involution of orthogonal type.
	\end{itemize} 
	Then $\mathrm{rdim}(X)\leq \mathrm{ind}(A)-1$. Moreover, if $\mathrm{ind}(A)\leq 3$, then $\mathrm{rdim}(X)= \mathrm{ind}(A)-1$. In particular $\mathrm{mcd}(X)\leq \mathrm{ind}(A)-1$.
\end{thm}
\begin{proof}
	Note that we have a semiorthogonal decomposition (see \cite{BLUS}) 
	\begin{eqnarray}
	D^b(X)=\langle D^b(k),D^b(A),...,D^b(k),D^b(A),D^b(C(A,\sigma))\rangle. 
	\end{eqnarray}	
	Let us first consider the case where $X$ is a twisted quadric associated to a central simple $\mathbb{R}$-algebra $(A,\sigma)$ with involution of orthogonal type. In this case, $A\simeq M_m(\mathbb{H})$ and $C(A,\sigma)$ is either a central simple $\mathbb{C}$-algebra or the product of two central simple $\mathbb{R}$-algebras. If $C(A,\sigma)\simeq M_n(\mathbb{C})$ take $Y=\mathrm{Spec}(\mathbb{C})$. Then we see that $D^b(C(A,\sigma))$ can be embedded into $D^b(Y)$ as an admissible subcategory. If $C(A,\sigma)\simeq A_1\times A_2$, we have $D^b(C(A,\sigma))=D^b(A_1)\times D^b(A_2)$. So there is a semiorthogonal decomposition
	\begin{eqnarray*}
		D^b(X)=\langle D^b(k),D^b(A),...,D^b(k),D^b(A),D^b(A_1), D^b(A_2)\rangle. 
	\end{eqnarray*}
	Since $A_1$ and $A_2$ are central simple $\mathbb{R}$-algebras, it is easy to check that there are smooth projective connected varieties $Y_1$ and $Y_2$ of dimension at most one such that $D^b(A_i)$ is embedded into $D^b(Y_i)$ as an admissible subcategory. And since $D^b(M_m(\mathbb{H}))$ is admissible in $D^b(Z)$, where $Z$ is the Brauer--Severi curve associated to $\mathbb{H}$, we conclude $\mathrm{rdim}(X)\leq 1=\mathrm{ind}(A)-1$. Now let $X$ be a twisted quadric associated to a central simple algebra $(A,\sigma)$ with involution of orthogonal type having trivial discriminant $\delta(A,\sigma)$. In this case, $C(A,\sigma)$ is the product $A_1\times A_2$ of two central simple $k$-algebras. So there is a semiorthogonal decomposition
	\begin{eqnarray*}
		D^b(X)=\langle D^b(k),D^b(A),...,D^b(k),D^b(A),D^b(A_1), D^b(A_2)\rangle. 
	\end{eqnarray*}
	After base change to some splitting field $L$ of $A$, this semiorthogonal decomposition gives rise to the full exceptional collection from Kapranov \cite{KA2S} on the smooth quadric $X\otimes_k L$.

	This in particular implies that $D^b(A_1)\otimes_k L\simeq D^b(L)$. The same for $A_2$. According to \cite{ABS}, Corollary 1.19 this implies that $A_1$ and $A_2$ are split by $L$. Now take $L$ to be a separable splitting field of degree $\mathrm{ind}(A)$. Let $Z_i$ be the Brauer--Severi variety associated to $D_i$, where $D_i$ is the central division algebra for which $A_i\simeq M_{l_i}(D_i)$ according to the Wedderburn theorem. Now we know that $D^b(A_i)$ can be embedded as an admissible subcategory into $D^b(Z_i)$. And since $\mathrm{ind}(A_i)\leq \mathrm{ind}(A)$, we obtain $\mathrm{dim}(Z_i)\leq \mathrm{ind}(A)-1$ for $i=1,2$. This gives $\mathrm{rdim}(X)\leq \mathrm{ind}(A)-1$. For the second statement of our theorem, we first consider the split case. So if $A$ is split, i.e. if $\mathrm{ind}(A)=1$, we have $\mathrm{rdim}(X)=0$ according to Proposition 6.10 and Proposition 5.4. Therefore, $\mathrm{rdim}(X)=0=\mathrm{ind}(A)-1$. If $\mathrm{ind}(A)=2$, the central simple algebra is non-split. So Propositions 5.4 and 6.10 imply $1\leq \mathrm{rdim}(X)$. On the other hand, the first statement of our theorem gives $\mathrm{rdim}(X)\leq \mathrm{ind}(A)-1=1$. Hence $\mathrm{rdim}(X)=0=\mathrm{ind}(A)-1$. Finally, consider the case $\mathrm{ind}(A)=3$. Then $\mathrm{rdim}(X)\leq \mathrm{ind}(A)-1=2$. In what follows, we exclude $\mathrm{rdim}(X)\in\{0,1\}$. Since $A$ is non-split, we immediately have $\mathrm{rdim}(X)\neq 0$. Now assume $\mathrm{rdim}(X)=1$. From \cite{AB1S}, Proposition 6.1.6 and 6.1.10 we conclude that if $\mathrm{rdim}(X)=1$ there must be a semiorthogonal decomposition of $D^b(X)$ whose components are either $D^b(K)$, where $K/k$ is a (finite) separable extension, or $D^b(D)$, where $D$ is a central division algebra with $\mathrm{ind}(D)\in\{1,2\}$, or $D^b(C)$, where $C$ is a smooth $k$-curve of positive genus. Note that $D^b(C)$ cannot be present because $K_0(X)$ is torsion free. Using the described semiorthogonal decomposition (that we get assuming $\mathrm{rdim}(X)=1$), we see that the noncommutative motive $U(\mathrm{perf}_{dg}(X))$ decomposes as
	\begin{eqnarray*}
		U(\mathrm{perf}_{dg}(X))\simeq\bigoplus^n_{i=1}U(K_i)\oplus \left(\bigoplus^m_{j=1}U(D_j)\right)
	\end{eqnarray*}
	for suitable central division algebras $D_j$ with $\mathrm{ind}(D_j)\in\{1,2\}$ and suitable separable extensions $K_i$.
	Note that $K_0(X)\simeq \mathbb{Z}^{\oplus r}$ and therefore $n+m=r$. After base change to $\bar{k}$ we also have $K_0(X_{\bar{k}})\simeq \mathbb{Z}^{\oplus r}$. %Let us denote the base change of the components $\mathcal{A}_i$ by $\bar{\mathcal{A}}_i$. 
	Now for any $D^b(K_i)$ we obtain after base change $D^b(K_i)_{\bar{k}}\simeq \prod^{r_i}_{q=1}D^b(\bar{k})$.
	%The base change of (2) to $k^s$ gives for the components $\mathcal{A}_i$ for which $\mathcal{A}_i\simeq D^b(K_i)$, $\bar{\mathcal{A}}_i\simeq \prod^{r_i}_{q=1}D^b(k^s)$.
	Hence $\sum^n_{i=1}{r_i}+m=r$. But this implies $\sum^n_{i=1}{r_i}=n$ and therefore $r_i=1$. Note that the semiorthogonal decomposition (21) forces the noncommutative motive to decompose as
	\begin{eqnarray*}
		U(\mathrm{perf}_{dg}(X))\simeq \bigoplus_{s}U(k)\oplus \bigoplus_{t}U(A)\oplus U(C(A,\sigma)). 
	\end{eqnarray*} 
	Since we assumed $\mathrm{ind}(A)=3$, $C(A,\sigma)$ is the product of two central simple algebras $A_1$ and $A_2$ (note that case (ii) cannot appear since in this case index is two). But then we have
	\begin{eqnarray*}
		\bigoplus^n_{i=1}U(k)\oplus \left(\bigoplus^m_{j=1}U(D_j)\right) \simeq \bigoplus_{s}U(k)\oplus \bigoplus_{t}U(A)\oplus U(A_1)\oplus U(A_2))
	\end{eqnarray*} 
	Then by Theorem 4.3 we conclude that $A$ must be split or that $A$ must be Brauer-equivalent to $D_j$ for some $j$. This contradicts $\mathrm{ind}(A)=3$. This completes the proof.
\end{proof}

%Summarizing the previews results we obtain:
%\begin{thm}

%\end{thm}

%\begin{proof}
%We proof the statement only for the case $X={_{\gamma_1}}Q_1\times {_{\gamma_2}}Q_2$ as the general case follows easily from induction. So assume $X$ admits a full exceptional collection. From Corollary 5.5 we obtain that the corresponding isotropic involution algebras $(A_1,\sigma_1)$ and $(A_2,\sigma_2)$ must be split, i.e ${_{\gamma_1}}Q_1$ and ${_{\gamma_2}}Q_2$ are smooth isotropic quadrics. Proposition 4.3 of \cite{DT} implies that both ${_{\gamma_1}}Q_1$ and ${_{\gamma_2}}Q_2$ have a $k$-rational point. Hence ${_{\gamma_1}}Q_1$ and ${_{\gamma_2}}Q_2$ are rational over $k$ (see \cite{CO}). This means that $X$ is rational over $k$. 

%Now assume $X$ is rational over $k$. Then $X$ admits a $k$-rational point. Applying the Lang--Nishimura Theorem to the projections $X\rightarrow {_{\gamma_1}}Q_1$ and $X\rightarrow {_{\gamma_2}}Q_2$ provides us with $k$-rational points on ${_{\gamma_1}}Q_1$ and ${_{\gamma_2}}Q_2$. This implies that ${_{\gamma_1}}Q_1$ and ${_{\gamma_2}}Q_2$ are smooth isotropic quadrics. But as mentioned earlier, the product of two smooth quadrics admits a full exceptional collection. This proves the equivalence of (i) and (ii). The equivalence of (ii) and (iii) is left to the reader.  
%\end{proof}

{\small MATHEMATISCHES INSTITUT, HEINRICH--HEINE--UNIVERSIT\"AT 40225 D\"USSELDORF, GERMANY}\\
E-mail adress: novakovic@math.uni-duesseldorf.de

\end{document}